\pdfoutput=1
\documentclass{amsart}
\usepackage{todonotes}
\usepackage{graphicx, amsmath, amssymb, amsthm}
\usepackage{tikz-cd}
\usepackage{spectralsequences,bbm}
\usepackage{hyperref} 
\usepackage[T1]{fontenc}
\hypersetup{
  colorlinks   = true, 
  urlcolor     = blue, 
  linkcolor    = blue, 
  citecolor   = red 
}
\usepackage{adjustbox}
\usepackage[normalem]{ulem}
\usepackage{contour}
\usepackage{mathtools} 

\contourlength{0.8pt}

\newcommand{\myuline}[1]{%
  \ifmmode
    \uline{\hphantom{#1}}
    \mathllap{\contour{white}{\ensuremath{#1}}}
  \else
    \uline{\phantom{#1}}%
    \llap{\contour{white}{#1}}%
  \fi
}
\usepackage[backend=biber, style=numeric, maxbibnames=99, doi=true,url=false,
    giveninits=true, hyperref]{biblatex}
\renewbibmacro{in:}{}

\newcommand{\bpcm}[1]{BP^{(\!(C_4)\!)}\langle#1\rangle}

\newcommand{\bpgm}[1]{BP^{(\!(G)\!)}\langle#1\rangle}
\newcommand{\Fil}{\mathrm{Fil}}

\newcommand{\Z}{{\mathbb  Z}}

\newcommand{\R}{{\mathbb R}}
\newcommand{\F}{{\mathbb F}}

\DeclareMathOperator{\Ext}{Ext}

\DeclareMathOperator{\Fun}{Fun}

\newcommand{\m}[1]{{\protect\underline{#1}}}

\newcommand{\Sp}{\mathcal Sp}

\mathchardef\mhyphen=45

\numberwithin{equation}{section}

\newtheorem{theorem}{Theorem}[section]
\newtheorem{lemma}[theorem]{Lemma}
\newtheorem{corollary}[theorem]{Corollary}

\newtheorem{proposition}[theorem]{Proposition}
\newtheorem{conjecture}[theorem]{Conjecture}

\newtheorem*{theorem*}{Theorem}
\newtheorem*{proposition*}{Proposition}

\theoremstyle{remark}
\newtheorem{remark}[theorem]{Remark}
\newtheorem{example}[theorem]{Example}

\theoremstyle{definition}

\newtheorem{definition}[theorem]{Definition}

\begin{document}

\title{Slice spectral sequences through synthetic spectra}
\author{Christian Carrick}
\address{Mathematisches Institut, Universität Bonn, Bonn, Germany}
\email{carrick@math.uni-bonn.de}
\thanks{The author was supported by the National Science Foundation under Grant No. DMS-2401918.}

\begin{abstract}
    We define a $t$-structure on the category of filtered $G$-spectra such that for a Borel $G$-spectrum $X$ the slice filtration of $X$ is the connective cover of the homotopy fixed-point filtration of $X$. Using this, we show that the slice spectral sequence for the norm $N_{C_2}^GMU_{\mathbb{\R}}$ of Real bordism theory refines canonically to a $\mathbb{E}_\infty$-algebra in $MU$-synthetic spectra, when $G$ is a cyclic $2$-group. Concretely, this gives a map of multiplicative spectral sequences from the classical Adams--Novikov spectral sequence of $\mathbb{S}$ to the slice spectral sequence for $N_{C_2}^GMU_{\mathbb{\R}}$ that respects the higher $\mathbb{E}_\infty$ structure, such as Toda brackets and power operations.

    We give a conjecture on the existence of vanishing lines in the equivariant Adams--Novikov spectral sequence based at tom Dieck's homotopical complex bordism $MU_G$. Conditional on this conjecture, our $t$-structure implies that the slice filtration for $N_{C_2}^GMU_{\mathbb{\R}}$ lifts further to an $\mathcal{O}$-algebra in $MU_G$-synthetic spectra, where $\mathcal{O}$ is the $\mathbb{N}_\infty$-operad with all norms from nontrivial subgroups of $G$.
\end{abstract}

\maketitle
\tableofcontents
\section{Introduction}
The equivariant slice spectral sequence was developed by Hill--Hopkins--Ravenel in their solution to the Kervaire invariant one problem \cite{HHR}, and it has since been at the center of a renaissance in both computations and theory in equivariant stable homotopy. The presence of norms in the slice spectral sequence played an essential role in their solution, and codifying this structure has motivated the study of $G$-symmetric monoidal structures \cite{hillhopkins}\cite{nardinshah}, norms in motivic homotopy theory \cite{bachmannhoyois}, and parametrized higher category theory \cite{barwickintro}. The slice spectral sequence brought about an explosion of progress in computations in chromatic homotopy at the prime $2$, with applications ranging from  Picard groups \cite{BBHS}\cite{HLS} to the Segal conjecture \cite{msz}\cite{carrickcofree}.

The slice spectral sequence was intended to give a more tractable analogue of the homotopy fixed point spectral sequences of higher real $K$-theories. That is, it gives a sort of \emph{connective model} of the latter with good finiteness and sparseness properties, vanishing "gap" regions, and strong vanishing lines on the $E_2$-page. Our first result makes this precise. Here, $\Fil(\Sp^G)$ is the functor $\infty$-category $\Fun(\Z^{\mathrm{op}},\Sp^G)$, the group $G$ is finite, and $\mathrm{HFP}(X)$ and $\mathrm{Slice}(X)$ are equivariant filtrations implementing the homotopy fixed point and slice spectral sequences of a $G$-spectrum $X$, respectively. We postpone the precise definition of these objects to Section \ref{sec:longintro}.

\begin{theorem}\label{thm:introthmconnectivecover}
    There is an accessible $t$-structure $\tau_{\ge0}^{\mathrm{slice}}$ on $\Fil(\Sp^G)$ that is compatible with the $G$-symmetric monoidal structure such that, for any $X\in \Sp^G$,
    \[\mathrm{Slice}(F({EG}_+,X))\simeq\tau_{\ge0}^{\mathrm{slice}}\mathrm{HFP}(X).\]
\end{theorem}

We call the $t$-structure of Theorem \ref{thm:introthmconnectivecover} the \emph{slice $t$-structure}. An immediate consequence is that when $X\simeq F({EG}_+,X)$ (i.e. when $X$ is Borel complete), the slice filtration (and, in particular, the slice spectral sequence) of $X$ is determined functorially by the homotopy fixed point filtration of $X$. This includes many examples of interest, such as Atiyah's Real $K$-theory $K_\R$ and the Fujii--Landweber Real bordism spectrum $MU_\R$. Our main interest is in the slice spectral sequence for the norms 
\[MU^{(\!(G)\!)}:=N_{C_2}^GMU_\R\]
of Real bordism theory for $G$ a cyclic $2$-group. In \cite{carrickcofree}, we showed that these $G$-spectra are Borel complete, so we have that
\[\mathrm{Slice}(MU^{(\!(G)\!)})\simeq\tau_{\ge0}^{\mathrm{slice}}\mathrm{HFP}(MU^{(\!(G)\!)}).\]

Theorem \ref{thm:introthmconnectivecover} contributes to a profitable shift in perspective in homotopy theory, whereby one replaces a study of \emph{spectral sequences} with a study of \emph{filtrations} that correspond to them. This perspective shift is discussed well by Antieau \cite{antieau} and van Nigtevecht \cite{sven}, for example. The idea is that replacing spectral sequences with filtrations gives one the sort of flexibility gained by studying spectra over homology theories.

It is not difficult to lift the slice spectral sequence to the filtered level, and we show how to implement the slice filtration as a $G$-commutative monoid in $\Fil(\Sp^G)$ in Section \ref{sec:equivarianttstructures}. 

\begin{theorem}
    There is a $G$-symmetric monoidal structure on the $G$-$\infty$-category $\Fil(\Sp^G)$ of filtered $G$-spectra such that, for a $G$-commutative monoid $R$ in $\Sp^G$, the filtration $\mathrm{Slice}(R)$ is a $G$-commutative monoid in $\Fil(\Sp^G)$.
\end{theorem}

This already makes many tools available to the slice spectral sequence: one may use the $\tau$-formalism to efficiently record hidden extensions, exotic restrictions/transfers, differential stretching, and filtration jumps; one may form (co)limits of diagrams of slice filtrations in $\Fil(\Sp^G)$ to obtained "modified" slice filtrations; and one has the various versions of the \emph{synthetic Leibniz rule} for differentials, as in \cite[Section 2.5]{smf} and \cite[Section 6]{kervaire}. One may, moreover, use the $G$-symmetric monoidal $\infty$-category
\[\mathrm{Mod}_{\Fil(\Sp^G)}(\mathrm{Slice}(MU^{(\!(G)\!)}))\]
as a kind of synthetic category for slice spectral sequences of $MU^{(\!(G)\!)}$-modules, along the lines of \cite{mmf}.

However, we show that $\mathrm{Slice}(MU^{(\!(G)\!)})$ has far more structure. In what follows, we let $\mathrm{AN}(\mathbb S)$ be the Adams--Novikov filtration of the sphere, which may be implemented as an $\mathbb E_\infty$-algebra in $\Fil(\Sp)$ via décalage as in \cite{mmf}, or via the even filtration of Hahn--Raksit--Wilson \cite{even}. The $\infty$-category $\mathrm{Syn}_{MU}$ denotes Pstragowski's category of $MU$-synthetic spectra \cite{piotr}.

\begin{theorem}\label{thm:introthmsyntehtic}
    There is a map of $\mathbb{E}_\infty$-algebras 
    \[\mathrm{AN}(\mathbb S)\to \mathrm{Slice}(MU^{(\!(G)\!)})^G\]
    in $\Fil(\Sp)$. That is, there is a lift of $\mathrm{Slice}(MU^{(\!(G)\!)})^G$ along the forgetful functor
    \[\mathrm{CAlg}(\mathrm{Syn}_{MU})\to \mathrm{CAlg}(\Fil(\Sp)).\]
\end{theorem}

Given Theorem \ref{thm:introthmconnectivecover}, the proof of Theorem \ref{thm:introthmsyntehtic} is quite straightforward. By adjunction one must produce a map
\[\mathrm{inf}_e^G\mathrm{AN}(\mathbb S)\to \mathrm{Slice}(MU^{(\!(G)\!)})\]
in $\mathrm{CAlg}(\Fil(\Sp^G))$, where $\inf_e^G$ is the inflation functor, left adjoint to genuine fixed points $(-)^G$. One has a map
\[\mathrm{inf}_e^G\mathrm{AN}(\mathbb S)\to \mathrm{HFP}(MU^{(\!(G)\!)})\]
in $\mathrm{CAlg}(\Fil(\Sp^G))$ by work of the author, Davies, and van Nigtevecht \cite{CDvN}. It suffices then to check that $\mathrm{inf}_e^G\mathrm{AN}(\mathbb S)$ is connective in the slice $t$-structure, and this follows from the classical vanishing line of slope 1 in the ANSS.

Theorem \ref{thm:introthmsyntehtic} gives lifts of $\mathrm{Slice}(MU^{(\!(G)\!)})^H$ to $\mathrm{Syn}_{MU}$ for all $H\subset G$, compatible with restrictions and transfers. This makes $\mathrm{Slice}(MU^{(\!(G)\!)})$ a commutative algebra in a more structured category.

\begin{definition}
    Let $\mathrm{Syn}_{MU}^G$ be the $G$-$\infty$-category 
    \[\mathrm{Mack}_G(\mathrm{Syn}_{MU})\]
    of \emph{$MU$-synthetic spectral Mackey functors}.
\end{definition}

With this definition, Theorem \ref{thm:introthmsyntehtic} refines to the statement that there exists a lift of $\mathrm{Slice}(MU^{(\!(G)\!)})$ along the forgetful map
\[\mathrm{CAlg}(\mathrm{Syn}_{MU}^G)\to\mathrm{CAlg}(\Fil(\Sp^G)).\]
In fact, for any module $M$ over $MU^{(\!(G)\!)}$, $\mathrm{Slice}(M)$ is naturally a module over $\mathrm{Slice}(MU^{(\!(G)\!)})$, so one has a canonical lift of $\mathrm{Slice}(M)$ to $\mathrm{Syn}_{MU}^G$. That is, $\mathrm{Slice}(M)$ has the structure of an $MU$-synthetic spectral Mackey functor.

We would like to emphasize now the concrete consequences of Theorem \ref{thm:introthmsyntehtic}.

\begin{corollary}\label{cor:introcorollarymapfromANSS}
    There is a map of multiplicative spectral sequences 
    \[\mathrm{ANSS}(\mathbb{S})\to\mathrm{SliceSS}(MU^{(\!(G)\!)})\]
    from the Adams--Novikov spectral sequence of the sphere to the slice spectral sequence of $MU^{(\!(G)\!)}$ converging to the unit map
    \[\pi_*\mathbb{S}\to\pi_*^GMU^{(\!(G)\!)}.\]
\end{corollary}

This gives one considerable computational control over the slice spectral sequence. One can use known information about $\mathrm{ANSS}(\mathbb{S})$ to determine differentials and hidden extensions in $\mathrm{SliceSS}(MU^{(\!(G)\!)})$ (or in the slice spectral sequence of an $MU^{(\!(G)\!)}$-module like $\mathrm{SliceSS}(\bpgm{m})$) in low degrees and then propagate this information using the multiplicative structure; see the computation of the descent spectral sequence for $Tmf$ by the author, Davies, and van Nigtevecht \cite{smf} for an example of this. It also makes it possible to directly detect greek letter elements in $MU^{(\!(G)\!)}$ using the detection arguments of the author and Davies in \cite{j1}\cite{j2prime3}\cite{j2prime2}. We intend to return to such applications in future work.

We also use Theorem \ref{thm:introthmconnectivecover} to connect the slice spectral sequence to the equivariant Adams--Novikov spectral sequence. Following Hausmann's proof \cite{hausmann} of Greenlees' conjecture on the homotopy groups $\pi_*MU_G$ of tom Dieck's homotopical complex bordism, there has been a surge of interest in developing equivariant chromatic homotopy via $MU_G$; see \cite{behrenscarlisle} and \cite{markuslennart} for example. In particular, there is an interest in understanding the Adams--Novikov spectral sequence in $\Sp^G$ based at $MU_G$. In Section \ref{sec:longintro}, we give two conjectures on this spectral sequence related to the existence of vanishing lines on the $E_2$-page. Given these conjectures, one has structured $MU_G$-synthetic lifts of $\mathrm{Slice}(MU^{(\!(G)\!)})$, in a precise sense which we discuss in Section \ref{sec:longintro}. For now, we state the concrete consequences of these conjectures in terms of the corresponding spectral sequences.

\begin{theorem}\label{thm:introthmintegergradedANSSG}
    Assuming Conjecture \ref{conjectureweak}, there is a map of integer graded multiplicative spectral sequences
    \[
    \mathrm{ANSS}_G(\mathbb{S})\to\mathrm{SliceSS}(MU^{(\!(G)\!)})
    \]
    converging to the unit map
    \[\pi_*^G\mathbb{S}\to\pi_*^GMU^{(\!(G)\!)}.\]
\end{theorem}

 \begin{theorem}\label{thm:introthmROANSSG}
    Assuming Conjecture \ref{conjecturemain}, there is a map of multiplicative $RO(G)$-graded spectral sequences
    \[\mathrm{ANSS}_G(\mathbb
    S)\to \mathrm{SliceSS}(MU^{(\!(G)\!)})\]
    converging to the unit map
    \[\pi_\bigstar^G\mathbb{S}\to \pi_\bigstar^GMU^{(\!(G)\!)}.\]
\end{theorem}

 

\subsection{Acknowledgments}
We would like to thank Lucas Piessevaux, Emma Brink, William Balderrama, Markus Hausmann, Lennart Meier, Kaif Hilman, Jack Davies, Sven van Nigtevecht, Ryan Quinn, Mike Hill, Branko Juran, and Natalie Stewart for helpful comments and conversations. We thank especially Lucas Piessevaux for explaining the $RO(G)$-graded Whitehead tower and the corresponding implementation of the $MU_G$-based ANSS; and we thank Kaif Hilman for explaining his work on geometric fixed points in a $G$-$\infty$-categorical context.

\subsection{Notation and conventions}
\begin{enumerate}
    \item We use the language of filtered spectra and filtrations, in particular the $\tau$-formalism. We recommend one reads about the $\tau$-formalism and filtered spectra in \cite{sven}.
    \item For a filtration F of an object $X$, we use the notation $\mathrm{F}(X)$ and $\mathrm{FSS}(X)$ to denote the filtration and corresponding spectral sequence, respectively. That is, $\mathrm{FSS}(X)$ denotes the $\tau$-Bockstein SS corresponding to $\mathrm{F}(X)$. For example, $\mathrm{AN}(\mathbb{S})$ denotes the Adams--Novikov filtration of $\mathbb S$, and $\mathrm{ANSS}(\mathbb{S})$ denotes the Adams--Novikov spectral sequence of $\mathbb S$.
    \item For a filtered object $X:\Z^{\mathrm{op}}\to\mathcal{C}$ and an integer $n$, we let $X^n$ denote the value of the functor $X$ at $n$. We denote the $n$-th associated graded of $X$ by
    \[\mathrm{gr}^nX:=(X/\tau)^n.\]
    \item For a filtered spectrum $X$ and a class $a\in \pi_{*,*}X$, we denote the mod $\tau$ projection of $a$ by $[a]$.
    \item For a filtered spectrum $X$, we denote by $\delta$ the boundary map in the multiplication by $\tau$ cofiber sequence
    \[\Sigma^{0,1}X\xrightarrow{\tau}X\xrightarrow{[-]}X/\tau\xrightarrow{\delta}\Sigma^{1,-2}X.\]
    See \cite{smf} for a discussion of this boundary map and its relationship to differentials in the $\tau$-Bockstein SS of $X$.
    \item We say a filtered spectrum $X$ is \emph{complete} if $\varprojlim X^n=0$.
    \item For a $G$-spectrum $X$, we use $\mathrm{Slice}^nX\to X$ to denote the $n$-th slice cover of $X$ (usually denote $P_nX\to X$). We use $\mathrm{Wh}^nX\to X$ to denote the $n$-th equivariant Whitehead cover of $X$ (usually denote $\tau_{\ge n}X\to X$), i.e. the $n$-th cover with respect to the standard $t$-structure on $X$ given by Mackey functor homotopy groups.
    \item The bigraded sphere $\mathbb{S}^{n,s}$ is the filtered spectrum
    \[
    \cdots\to0\to \mathbb{S}^n=\mathbb{S}^n=\cdots
    \]
    starting at component $n+s$, so that for $X\in\Fil(\Sp)$
    \[\pi_{n,s}X:=[\mathbb{S}^{n,s},X]=\pi_nX^{n+s}.\]
    \item Similarly, for $V$ a $G$-representation, let $\mathbb{S}^{V,s}$ be the filtered $G$-spectrum
    \[
    \cdots\to0\to \mathbb{S}^V=\mathbb{S}^V=\cdots
    \]
    starting at component $\dim(V)+s$, so that for $X\in\Fil(\Sp^G)$
    \[\pi_{V,s}X:=[\mathbb{S}^{n,s},X]=\pi_VX^{\dim(V)+s}.\]
    \item For a subgroup $H\subset G$ and a normal subgroup $N\subset G$, we use the notation $\mathrm{Res}_H^G$, $\mathrm{Ind}_H^G$, $N_H^G$, $\Phi^H$, $(-)^H$, and $\mathrm{inf}_{G/N}^G$ to denote the restriction, induction, norm, geometric fixed points, genuine fixed points, and inflation functors on $\Sp^G$. In a filtered equivariant context, apart from $N_H^G$, these functors are applied pointwise. The norm $N_H^G$ in these contexts comes from an implicit $G$-symmetric monoidal structure.
    \item For a set of objects $S$ in a presentable stable $\infty$-category $\mathcal{C}$, the notation $\langle S\rangle$ denotes the smallest subcategory of $\mathcal{C}$ containing $S$ and closed under colimits and extensions.
    \item The group $G$ will always be finite. When discussing the norms $MU^{(\!(G)\!)}$, the group $G$ will always be a cyclic $2$-group.
    
\end{enumerate}

\section{Outline of results and methods}\label{sec:longintro}

Constructing the map of spectral sequences
\[\mathrm{ANSS}(\mathbb S)\to\mathrm{SliceSS}(MU^{(\!(G)\!)})\]
of Corollary \ref{cor:introcorollarymapfromANSS} directly seems quite difficult. The two spectral sequences arise from completely different constructions; one is an Adams spectral sequence arising from a descent object and the other is more like an Atiyah--Hirzebruch spectral sequence, arising from a sequence of $t$-structures on $\Sp^G$. We begin with a discussion of why, then, one would expect such a map to exist, and use this to motivate our definition of the slice $t$-structure.

\subsection{The slice spectral sequence and connective models} Intuition from Ravenel's solution to the Kervaire invariant problem for primes $p\ge 5$ \cite{ravodd} suggested that the $2$-primary Kervaire invariant problem could be solved by understanding the homotopy fixed point spectral sequences for certain Hopkins--Miller higher real $K$-theories. These theories, known also as the $EO_n$'s are formed by taking the homotopy fixed points of Morava $E$-theory with respect to finite subgroups of the Morava stabilizer group. These spectral sequences, however, can be extremely difficult to compute. They are not degreewise finitely generated, and in general possess no good vanishing lines or regions on the $E_2$-page. 

At height $2$, $EO_2$ is a $K(2)$-local model of the topological modular forms spectrum $TMF$. The computations of Hopkins--Mahowald showed that the Adams--Novikov spectral sequence for the connective topological modular forms spectrum $tmf$ was much more approachable than the HFPSS for $EO_2$ \cite{hopkinsmahowald}. In particular, $\mathrm{ANSS}(tmf)$ is a finite type, first quadrant spectral sequence with a vanishing line of slope 1 on the $E_2$-page. A general approach to producing connective models of $EO_n$ theories was thus of great interest, and Hill--Hopkins--Ravenel gave a technique to do this in certain cases via genuine equivariant homotopy theory. These connective theories, known as $\bpgm{m}$, were constructed as certain quotients of $MU^{(\!(G)\!)}$. See Sections 1 and 2 of \cite{carrickhill} for a discussion of these theories.

In contrast with $tmf$, producing the ANSS of these theories was not approachable and thus Hill--Hopkins--Ravenel developed the slice spectral sequence as a suitable replacement. While the Detection theorem of Hill--Hopkins--Ravenel is proved using $\mathrm{HFPSS}(\bpgm{m})$, the proofs of their Gap and Periodicity theorems crucially require $\mathrm{SliceSS}(\bpgm{m})$. In fact, the Gap theorem is directly visible on the $E_2$-page of the slice spectral sequence due to its vanishing regions. Similarly, the Periodicity theorem follows from a family of differentials proven inductively using the vanishing lines of slope $|G|-1$ on the $E_2$-page of the slice spectral sequence.

However, the slice spectral sequences of these theories do indeed behave much like the Adams--Novikov spectral sequences of theories like $tmf$. For example, direct computation shows that the slice spectral sequence of $k_\R=BP^{(\!(C_2)\!)}\langle 1\rangle$ is isomorphic (in integer degrees) to $\mathrm{ANSS}(ko)$. This leads to the question of whether an identification $\mathrm{SliceSS}(\bpgm{m})\cong\mathrm{ANSS}(\bpgm{m}^G)$ holds in general. The author and Hill \cite{MUhomology} showed that this does not hold for any other $\bpgm{m}$. Theorem \ref{thm:introthmsyntehtic} states, however, that $\mathrm{SliceSS}(\bpgm{m})$ gives a \emph{modified} Adams--Novikov spectral sequence for $\bpgm{m}$, in the sense of \cite[Section 2.2]{j1}.

\subsection{Linear $t$-structures and slice filtrations} The construction of the map of Corollary \ref{cor:introcorollarymapfromANSS} proceeds from a few well-known facts. First, there is a diagram of spectral sequences
\[\mathrm{SliceSS}(MU^{(\!(G)\!)})\to \mathrm{HFPSS}(MU^{(\!(G)\!)})\leftarrow\mathrm{ANSS}(\mathbb{S}).\]
The left map was constructed by Ullman \cite{ullman}, and the right map was constructed by Hill--Hopkins--Ravenel in \cite[Section 11.3.3]{HHR}, and in a more structured context by the author, Davies, and van Nigtevecht \cite{CDvN}. The second fact is that the left map induces an isomorphism on the $E_2$-page in the region $y\le x$, as shown by Ullman \cite{ullman}. The third fact is that $\mathrm{ANSS}(\mathbb{S})$ is concentrated in the region $y\le x$.

This led the author to suspect that the map of Corollary \ref{cor:introcorollarymapfromANSS} could be constructed using the linear $t$-structures on filtered spectra introduced in the context of synthetic spectra by the author and Davies in \cite{j1}, and  in a general filtered context by Lee--Levy \cite{leelevy}. For a real number $\alpha$, the $y=\alpha x$ linear $t$-structure has connective objects those filtered spectra $X$ such that $\pi_{x,y}X=0$ for $y>\alpha x$. We review these $t$-structures in Section \ref{sec:lineartstructures} and prove a sort of omnibus theorem that completely determines the $\tau$-Bockstein SS of a linear connective cover of a filtered spectrum $X$ from the $\tau$-Bockstein SS of $X$.

We implement this idea by defining a slice filtration internal to $\Fil(\Sp^G)$. The slice filtration on $\Sp^G$ may be defined in terms of \emph{geometric connectivity}, as shown by Hill--Yarnall \cite{hillyarnall}. From this point of view, slice connectivity is defined by pulling back the standard $t$-structure on $\Sp$ along the geometric fixed-point functors. That is, we say a $G$-spectrum $X$ is slice $\ge n$ if 
\[\Phi^HX\ge\lceil n/|H|\rceil\]
for all $H\subset G$. Given any $t$-structure on $\Fil(\Sp)$, one may mimic this definition of the slice filtration to define a slice filtration on $\Fil(\Sp^G)$, via the geometric fixed point functors $\Phi^H:\Fil(\Sp^G)\to\Fil(\Sp)$. Now the target $\Fil(\Sp)$ has many interesting $t$-structures, such as the linear $t$-structures discussed above. We let $\mathcal{O}$ denote the $\mathbb N_\infty$-operad with all norms from nontrivial subgroups.

\begin{definition}\label{def:introdefslicestructures}
    The \emph{slice $t$-structure} on $\Fil(\Sp^G)$ has connective objects those $X$ such that $\Phi^HX$ is connective in the $y=(|H|-1)x$ linear $t$-structure on $\Fil(\Sp)$, for all $H\subset G$. The $\mathcal{O}$-\emph{slice $t$-structure} on $\Fil(\Sp^G)$ has connective objects those $X$ such that $\Phi^HX$ is connective in the $y=(|H|-1)x$ linear $t$-structure on $\Fil(\Sp)$, for all $\{e\}\neq H\subset G$. 
\end{definition}

To work with these $t$-structures and understand their monoidality properties, we revisit the notion of a slice filtration on a category, as well as in a $G$-$\infty$-categorical context, in Sections \ref{sec:equivariantfiltrations} and \ref{sec:equivarianttstructures}. The $\mathcal{O}$-slice $t$-structure is a sort of auxiliary $t$-structure used to deal with the fact that $\mathrm{inf}_e^G\mathrm{AN}(\mathbb S)$ is not connective in the slice $t$-structure, but is connective in the $\mathcal{O}$-slice $t$-structure. This makes no difference in the context of Theorem \ref{thm:introthmsyntehtic} since the slice filtration is the connective cover of the homotopy fixed point filtration in \emph{both} $t$-structures (see Lemma \ref{lem:tstructuresagreeonhfpss}). This comes from the fact that the two filtrations agree on the underlying spectra.

The slice $t$-structure has the property that the slice filtration $\mathrm{Slice}(X)$ of a $G$-spectrum $X$ is obtained by taking the constant filtration $\mathrm{const}(X)$ and then taking the connective cover with respect to the slice $t$-structure. One may take this as a definition of $\mathrm{Slice}(X)$ and the pointwise Borel completion $F({EG}_+,\mathrm{Slice}(X))$ as a definition of $\mathrm{HFP}(X)$. Since the slice $t$-structure is defined so analogously to the ordinary slice filtration, Theorem \ref{thm:introthmconnectivecover} is nearly a tautology with the definitions in place; see Lemma \ref{lemma:sliceandhfpss}.

The proof of Theorem \ref{thm:introthmsyntehtic} is now straightforward. 

\begin{proof}[proof of Theorem \ref{thm:introthmsyntehtic}]
    The classical $y=x$ vanishing line in $\mathrm{ANSS}(\mathbb S)$ implies that $\mathrm{AN}(\mathbb S)$ is connective in the $y=x$ linear $t$-structure on $\Fil(\Sp)$, hence so is 
    \[\Phi^H\mathrm{inf}_e^G\mathrm{AN}(\mathbb S)\simeq \mathrm{AN}(\mathbb S),\]
    and therefore $\mathrm{inf}_e^G\mathrm{AN}(\mathbb S)$ is connective in the $\mathcal{O}$-slice $t$-structure.
\end{proof}

\begin{remark}
    Theorem \ref{thm:introthmsyntehtic} produces $\mathrm{Slice}(MU^{(\!(G)\!)})$ as an $\mathbb E_\infty$-algebra in $\mathrm{Syn}_{MU}^G$. There is no hope of lifting this structure further to that of a $G$-$\mathbb E_\infty$-algebra or an $\mathcal{O}$-algebra because the category $\mathrm{Syn}_{MU}^G$ is not $G$-symmetric monoidal (or even $\mathcal{O}$-monoidal). Indeed, if it were, then $\mathrm{inf}_e^G\mathrm{AN}(\mathbb S)$ would have the structure of a $G$-$\mathbb E_\infty$-algebra in $\Fil(\myuline{\Sp})$ (see Proposition \ref{prop:syntheticspectralmackey}), and therefore $\mathrm{gr}^0\mathrm{inf}_e^G\mathrm{AN}(\mathbb S)$ would have the structure of a $G$-$\mathbb E_\infty$-algebra in $\myuline{\Sp}$. However, it is straightforward to see that this $\mathrm{gr}^0$ is $\mathrm{inf}_e^G(\Z)$, which in general does not possess norms (see \cite[Example 3.28]{derivedmackey} for example).
\end{remark}

\subsection{Maps from the $MU_G$-based ANSS} One may implement the (integer-graded) $MU_G$-based Adams spectral sequence in $\Fil(\Sp^G)$ via décalage
\[\mathrm{AN}_G(\mathbb{S}):=\mathrm{Tot}(\mathrm{Wh}(MU_G)\implies\mathrm{Wh}(MU_G\otimes MU_G)\Rrightarrow\cdots),\]
as in \cite{mmf}, where $\mathrm{Wh}$
 is the equivariant Whitehead tower. In Section \ref{sec:slicefiltrationisconnectivecover}, we construct a map of $\mathbb{E}_\infty$-algebras $\mathrm{AN}_G(\mathbb{S})\to\mathrm{HFP}(MU^{(\!(G)\!)})$. We also show that the $\mathcal{O}$-slice connectivity of $\mathrm{AN}_G(\mathbb{S})$ follows from the following conjecture on vanishing lines.

\begin{conjecture}\label{conjectureweak}
    For $G=C_{2^n}$, the integer graded $E_2$-page of the $MU_G$-based ASS of the sphere vanishes above the line $y=(2^n-1)x$. That is
    \[\Ext^{s,t}_{({MU_G}_*,{MU_G}_*{MU_G})}({MU_G}_*,{MU_G}_*)=0\]
    for $s>\frac{2^n-1}{2^n}t$.
\end{conjecture}

\begin{theorem}
    Given Conjecture \ref{conjectureweak}, there is a map
    \[\mathrm{AN}_G(\mathbb{S})\to \mathrm{Slice}(MU^{(\!(G)\!)})\]
    in $\mathrm{CAlg}(\Fil(\Sp^G))$.
\end{theorem}

The module category $\mathrm{Mod}_{\Fil(\Sp^G)}(\mathrm{AN}_G(\mathbb S))$ gives an \emph{ad hoc} definition of a sort of "integer-graded" $MU_G$-synthetic category. The above conjecture would thus provide an $MU_G$-synthetic lift of $\mathrm{Slice}(MU^{(\!(G)\!)})$ in this sense.

\begin{remark}\label{rmk:withoutconjecture}
    One can always simply "force" Conjecture \ref{conjectureweak} to hold by passing to a connective cover. That is, there is a map
    \[\tau_{\ge0}^{\mathcal O}\mathrm{AN}_G(\mathbb S)\to \mathrm{Slice}(MU^{(\!(G)\!)})\]
    in $\mathrm{CAlg}(\Fil(\Sp^G))$, where $\tau_{\ge0}^{\mathcal O}$ is the connective cover with respect to the $\mathcal{O}$-slice $t$-structure. This is not conditional on Conjecture \ref{conjectureweak}, and may be in principle useful since $\tau_{\ge0}^{\mathcal O}\mathrm{AN}_G(\mathbb S)$ is still a filtration of the equivariant sphere spectrum whose associated spectral sequence is determined from that of $\mathrm{AN}_G(\mathbb S)$ via the omnibus theorem of Section \ref{sec:lineartstructures}.
\end{remark}

\subsection{$RO(G)$-indexed filtrations} To implement the full $RO(G)$-graded $MU_G$-based Adams--Novikov spectral sequence in a filtered setting, one must work instead with $RO(G)$-filtered $G$-spectra; that is, in the functor category $\Fun(RO(G)^{\mathrm{op}},\Sp^G)$. Here one sets 
\[\mathrm{AN}^{RO}_G(\mathbb{S}):=\mathrm{Tot}(\mathrm{Wh}^{RO}(MU_G)\implies\mathrm{Wh}^{RO}(MU_G\otimes MU_G)\Rrightarrow\cdots),\]
where $\mathrm{Wh}^{RO}$ is an $RO(G)$-filtered Whitehead tower functor defined by
\[(\mathrm{Wh}^{RO}(X))^V=\Sigma^V\tau_{\ge0}\Sigma^{-V}X,\]
for $V\in RO(G)$. This definition is due to Lucas Piessevaux, and we thank them for introducing this to us.

In Sections \ref{sec:equivariantfiltrations} and \ref{sec:equivarianttstructures}, we set up these notions carefully. In Section \ref{sec:equivariantfiltrations}, we study a useful adjunction
\[
\begin{tikzcd}
\mathrm{Total}:\Fun(RO(G)^{\mathrm{op}},\Sp^G)\arrow[r,shift left]&\Fil(\Sp^G):\mathrm{dim}^*\arrow[l,shift left]    
\end{tikzcd}
\]
where $\mathrm{dim}^*$ is the pullback along the dimension function $\mathrm{dim}:RO(G)\to \Z$. We describe the left adjoint $\mathrm{Total}$ explicitly and show how it resembles the extraction of a total complex from a double complex. We show that $\mathrm{dim}^*\mathrm{Slice}(X)$ and $\mathrm{dim}^*\mathrm{HFP}(X)$ implement the $RO(G)$-graded slice and homotopy fixed point spectral sequences of $X$, respectively. We show that $\mathrm{Total}$ refines to a $G$-symmetric monoidal left adjoint, and we construct a map 
\[\mathrm{AN}_G^{RO}(\mathbb S)\to \dim^*\mathrm{HFP}(MU^{(\!(G)\!)})\]
of $G$-commutative monoids in $\Fun(RO(G)^{\mathrm{op}},\Sp^G)$. This motivates the following conjecture.

\begin{conjecture}\label{conjecturemain}
    The filtered $G$-spectrum $\mathrm{Total}(\mathrm{AN}^{RO}_G(\mathbb{S}))$ is connective in the $\mathcal{O}$-slice $t$-structure on $\Fil(\Sp^G)$.
\end{conjecture}

We expect that this conjecture may be reduced to a concrete conjecture about vanishing lines on the $RO(G)$-graded pages of the $MU_G$-based ANSS, along the lines of Conjecture \ref{conjectureweak}. This reduction may be approached by filtering $\mathrm{Total}(\mathrm{AN}^{RO}_G(\mathbb{S}))$ in the manner that one filters a double complex to attain the associated spectral sequence. In the following, we let $\mathcal{O}$ denote the $\mathbb N_\infty$-operad possessing all norms from nontrivial subgroups of $G$, as before.

\begin{theorem}
    Assuming Conjecture \ref{conjecturemain}, there is a map
    \[
    \mathrm{AN}_G^{RO}(\mathbb S)\to \mathrm{dim}^*\mathrm{Slice}(MU^{(\!(G)\!)})
    \]
    of $\mathcal{O}$-algebras in $\Fun(RO(G)^{\mathrm{op}},\Sp^G)$.
\end{theorem}

\begin{remark}
As in Remark \ref{rmk:withoutconjecture}, one may in a sense force Conjecture \ref{conjecturemain} to be true and obtain a map. In particular, one may define a $t$-structure on $\Fun(RO(G)^{\mathrm{op}},\Sp^G)$ as having connective objects those $X$ with $\mathrm{Total}(X)$ connective in the $\mathcal{O}$-slice $t$-structure on $\Fil(\Sp^G)$. Letting $\tau_{\ge0}^{\mathrm{total}}$ denote the corresponding connective cover functor, there is a map
\[\tau_{\ge0}^{\mathrm{total}}\mathrm{AN}_G^{RO}(\mathbb S)\to \mathrm{dim}^*\mathrm{Slice}(MU^{(\!(G)\!)})\]
of $\mathcal O$-algebras in  $\Fun(RO(G)^{\mathrm{op}},\Sp^G)$, which is not conditional on Conjecture \ref{conjecturemain}.
\end{remark}

\begin{remark}
    In his proof of Greenlees' conjecture, Hausmann shows that the Hopf algebroid $(\pi_*\Phi^HMU_G,\pi_*\Phi^H(MU_G\otimes MU_G))$ is equivalent to $(MU_*,MU_*MU)$ for all $H\subset G$ \cite[Proposition 5.55]{hausmann}. This leads one to expect the formulas $\Phi^H\mathrm{AN}_G(\mathbb{S})\simeq \mathrm{AN}(\mathbb S)$ and $\Phi^H\mathrm{AN}_G^{RO}(\mathbb{S})\simeq \dim^*\mathrm{AN}(\mathbb S)$, from which Conjecture \ref{conjectureweak} and Conjecture \ref{conjecturemain} would follow immediately, respectively. However, these formulas do not appear to hold, and in particular do not follow from the above identification of Hopf algebroids. The issue is that the geometric fixed points functors do not pass through the décalage construction, and understanding the failure of this appears to be at the heart of understanding these conjectures.
\end{remark}

One may obtain an \emph{ad hoc} definition of $MU_G$-synthetic spectra as before by setting
\[\mathrm{Syn}_{MU_G}=\mathrm{Mod}_{\Fun(RO(G)^{\mathrm{op}},\Sp^G)}(\mathrm{AN}^{RO}_G(\mathbb{S})).\]
This is expected to agree with a definition of $MU_G$-synthetic spectra given in forthcoming work of Lucas Piessevaux. According to this definition, Conjecture \ref{conjecturemain} gives $\mathrm{dim}^*\mathrm{Slice}(MU^{(\!(G)\!)})$ the structure of an $\mathcal{O}$-algebra in $\mathrm{Syn}_{MU_G}$.

\subsection{Questions}
\begin{enumerate}
    \item Is there a similar $t$-structure on filtered motivic spectra that provides an $MU$-synthetic structure on the motivic slice spectral sequence for $MGL$?
    \item What sort of computational leverage can one gain from the $G$-commutative monoid structure on $\mathrm{Slice}(R)$ for $R$ a $G$-commutative monoid in terms of norms? Is there a $C_2$-equivariant refinement of the synthetic Leibniz rule of \cite[Section 2.2]{smf} for a $C_2$-commutative monoid in $\Fil(\Sp^{C_2})$ that states that
    \[\delta_n^{2n}(N_e^{C_2}(x))=\mathrm{tr}(u_\sigma^{-|x|}\gamma(x)\delta_n^{2n}(x))?\]
    \item Do the Meier--Shi--Zeng decompositions of the slice spectral sequence of $\bpgm{m}$ \cite{meiershizengtranschromatic} correspond to Tate square decompositions of $\mathrm{Slice}(\bpgm{m})$ in $\mathrm{Syn}_{MU}^G$, via Theorem \ref{thm:introthmsyntehtic}?
    \item Theorem \ref{thm:introthmsyntehtic} gives $\mathrm{Slice}(MU^{(\!(G)\!)})/\tau$ the structure of an $\mathbb E_\infty$-algebra in $\mathrm{Stable}_{MU_*MU}$. Does this algebra arise from a map of (classical) stacks
    \[\mathcal{M}_{MU^{(\!(G)\!)}}\to \mathcal{M}_{FG}?\]
    Can this be approached in terms of elliptic curves for $tmf_0(3)$ and $tmf_0(5)$, viewed as $\mathbb E_\infty$ forms of $BP_\R\langle 2\rangle$ and $\bpcm{1}$, respectively?
\end{enumerate}

\section{Equivariant filtered objects}\label{sec:equivariantfiltrations} In this section, we develop the basics of $\Z$- and $RO(G)$-filtered $G$-spectra, their $G$-symmetric monoidal structures, and the behavior of geometric fixed points. The results of this section are surely well-known to the experts, and we have only collected the basic tools we need in this article. In particular, forthcoming work of Kaif Hilman develops many of the notions in this section in more detail and in a broader context. We thank Kaif Hilman for discussing some of these results with us.

\subsection{$G$-symmetric monoidal $\infty$-categories and geometric fixed points} In his thesis, Nardin defined $G$-equivariant analogues of the $\infty$-categories $\mathrm{Pr}^L$ and $\mathrm{Pr}^L_{st}$ of Lurie, with many analogous properties \cite{nardinthesis}. We refer the reader also to the work of Hilman \cite{kaifpresentable} for these notions. As in the nonequivariant case, the $G$-category $\myuline{\Sp}$ of genuine equivariant spectra satisfies a universal property in $\myuline{\mathrm{Pr}}^L_{G,st}$. Namely, specifying a morphism $F:\myuline{\Sp}\to\underline{\mathcal{C}}$ in $\myuline{\mathrm{Pr}}^L_{G,st}$ is equivalent to identifying an object $F(\mathbb{S})\in\myuline{\mathcal{C}}(G/G)$. Moreover, $\myuline{\mathrm{Pr}}^L_{G}$ and $\myuline{\mathrm{Pr}}^L_{G,st}$ have canonical $G$-symmetric monoidal structures, and $\myuline{\Sp}$ is the unit in the latter. 

We fix now a $\myuline{\mathcal{C}}$ a presentably $G$-symmetric monoidal stable $\infty$-category, i.e. an object $\myuline{\mathcal{C}}\in\mathrm{CAlg}_G(\myuline{\mathrm{Pr}}_{G,st}^L)$. We may define a geometric fixed points functor on $\myuline{\mathcal{C}}$ as follows. By \cite[Theorem A]{kaifpresentable}, a presentable $G$-$\infty$-category is specified by a functor $\mathcal{O}_G^{\mathrm{op}}\to \mathrm{Pr}^L$, where $\mathcal{O}_G$ is the orbit category of $G$.

\begin{definition}\label{def:geometricfixedpointscategory}
Let $\Phi\myuline{\Sp}$ be the $G$-$\infty$-category defined by right Kan extension of the functor $*\xrightarrow{\Sp}\mathrm{Pr}^L_{st}$
    along the inclusion $*\xrightarrow{G/G}\mathcal{O}_G^{\mathrm{op}}$.
 Concretely, one has
    \[\Phi\myuline{\Sp}(G/H)=\begin{cases}
        \Sp&H=G\\ *&H\neq G
    \end{cases}\]
\end{definition}

Since $\Phi\myuline{\Sp}$ is fiberwise stable, by \cite[Definition 2.3.3]{nardinthesis}, $\Phi\myuline{\Sp}\in \myuline{\mathrm{Pr}}_{G,st}^L$ if it is $G$-semiadditive in the sense of \cite[Definition 4.1.2]{kaifpresentable}, i.e. if indexed coproducts agree with indexed products in $\Phi\myuline{\Sp}$. However, since $\Phi\myuline{\Sp}(G/H)=*$ for $H$ proper, indexed (co)products are trivial unless indexed over trivial finite $G$-sets. It follows that $\Phi\myuline{\Sp}$ is $G$-semiadditive since $\Sp$ is semiadditive, and therefore that $\Phi\myuline{\Sp}$ is $G$-stable. This allows for the following definition.

\begin{definition}\label{def:geometricfixedpointsfunctor}
    Let $\Phi^G:\myuline{\Sp}\to\Phi\myuline{\Sp}$ be the morphism in $\myuline{\mathrm{Pr}}_{G,st}^L$ defined by $\Phi^G(\mathbb{S})=\mathbb{S}\in \Sp$. We define the geometric fixed points functor for $\myuline{\mathcal{C}}$ via the tensor product
    \[
    \Phi^G_{\underline{\mathcal{C}}}:\myuline{\mathcal{C}}\simeq\myuline{\mathcal{C}}\otimes\myuline{\Sp}\xrightarrow{\mathrm{id}\otimes\Phi^G}\myuline{\mathcal{C}}\otimes\Phi\myuline{\Sp}
    \]
\end{definition}

\begin{lemma}\label{lemma:geometricfixedpointsatG/G}
        Evaluating $\Phi^G_{\underline{\mathcal{C}}}$ at $G/G$, one recovers the symmetric monoidal left adjoint
    \[\myuline{\mathcal{C}}(G/G)\simeq\myuline{\mathcal{C}}(G/G)\otimes_{\Sp^G}\Sp^G\xrightarrow{\mathrm{id}\otimes \Phi^G}\myuline{\mathcal{C}}(G/G)\otimes_{\Sp^G}\Sp\]
\end{lemma}
\begin{proof}
    This follows directly from the formula for the tensor product of $G$-presentable $\infty$-categories
\[\myuline{\mathcal{C}}\otimes\myuline{\mathcal{D}}\simeq\mathrm{Fun}^R(\myuline{\mathcal{C}}^{\mathrm{op}},\myuline{\mathcal{D}})\]
    of \cite[Proposition 6.7.5]{kaifpresentable}, and the fact that right adjoints $\myuline{\mathcal{C}}^{\mathrm{op}}\to \Phi\myuline{\Sp}$ are the same as right adjoints $\myuline{\mathcal{C}}(G/G)^{\mathrm{op}}\to \Sp$, by \cite[Corollary 2.2.7]{kaifpresentable}. 
\end{proof}

We will need the following fact about $\Phi^G_{\underline{\mathcal{C}}}$, which generalizes the equivalence of functors $\Phi^G\circ N_H^G\simeq\Phi^H:\Sp^H\to \Sp$ shown by Hill--Hopkins--Ravenel in \cite[Proposition 2.57]{HHR}.

\begin{proposition}\label{prop:geomfixedpointsofnorm}
    Let $H$ be a subgroup of $G$. The composition
    \[
    \myuline{\mathcal{C}}(G/H)\xrightarrow{N_H^G}\myuline{\mathcal{C}}(G/G)\xrightarrow{\Phi^G_{\underline{\mathcal{C}}}}\myuline{\mathcal{C}}(G/G)\otimes_{\Sp^G}\Sp
    \]
    is a symmetric monoidal left adjoint.
\end{proposition}
\begin{proof}
    The functors $N_H^G$ and $\Phi^G$ are symmetric monoidal so it remains to show that the composition preserves colimits. This follows from the distributivity of norms in $\myuline{\mathcal{C}}$ with respect to $G$-colimits, the fact that $\Phi^G_{\underline{\mathcal{C}}}$ is a $G$-left adjoint, and the fact that $G$-colimits in $\myuline{\mathcal{C}}\otimes\Phi\myuline{\Sp}$ are given by ordinary colimits. The first fact holds by assumption for $\myuline{\mathcal{C}}$ and expresses the composition of $N_H^G$ with an $H$-colimit diagram as in \cite[Recollection 2.1.31]{kaif}. The second fact is by definition of $\Phi^G_{\underline{\mathcal{C}}}$. The third fact follows as $(\myuline{\mathcal{C}}\otimes\Phi\myuline{\Sp})(G/H)=*$ for $H$ a proper subgroup, whereby one sees that ($G$-colimits of) $G$-diagrams $\myuline{I}\to \myuline{\mathcal{C}}\otimes\Phi\myuline{\Sp}$ are equivalent to (colimits of) diagrams $\myuline{I}(G/G)\to (\myuline{\mathcal{C}}\otimes\Phi\myuline{\Sp})(G/G)$.  
\end{proof}

\subsection{Equivariant filtered objects} 
We fix a $G$-symmetric monoidal poset $\myuline{I}$, i.e. a product preserving functor
\[\myuline{I}:\mathrm{Span}(\mathrm{Fin}_G)\to \mathrm{Posets}\]
from the Span category of the category of finite $G$-sets to the category of posets. Letting $\mathrm{Cat}$ denote the category of $\infty$-category of $\infty$-categories, the product preserving inclusion  $\mathrm{Posets}\hookrightarrow\mathrm{Cat}$ allows us to regard $\myuline{I}$ as a $G$-symmetric monoidal $\infty$-category. For a $G$-symmetric monoidal poset $\myuline{I}$, we denote by $\myuline{I}^\delta$ the composition
    \[\mathrm{Span}(\mathrm{Fin}_G)\xrightarrow{\underline{I}} \mathrm{Posets}\xrightarrow{\delta}\mathrm{Posets},\]
    where $\mathrm{Posets}\xrightarrow{\delta}\mathrm{Posets}$ is the (product-preserving) functor sending a poset to its underlying set regarded as a discrete poset.

\begin{example}
    We will be most interested in the following examples of $G$-symmetric monoidal posets:
    \begin{enumerate}
        \item $\myuline{\Z}$ and $\myuline{\Z}^\delta$, where $\myuline{\Z}$ is the constant Mackey functor at $\Z$ with the usual poset structure on the integers.
        \item $\myuline{RO}$, where $\myuline{RO}$ is the Mackey functor assigning the real representation ring $RO(H)$ to $G/H$, where for $H$ representations $V,W,V',W'$, $[V]-[W]\le[V']-[W']\in RO(H)$ if and only if there is an embedding $V\oplus W'\hookrightarrow V'\oplus W$ of $H$-representations. One checks this is well-defined using the semisimplicity of $\mathbb{R}[G]$. In fact, the latter fact gives an isomorphism of $RO(G)$ with a product of copies of $\Z$ indexed by the (isomorphism classes) of irreducible representations of $G$, and this definition makes this an isomorphism of posets.
        \item
        $\myuline{RU}$, where $\myuline{RU}$ is the Mackey functor assigning the complex representation ring $RU(H)$ to $G/H$, with poset structure defined as in the $\myuline{RO}$ case.
    \end{enumerate}
\end{example}

For $\myuline{\mathcal{C}}$ a presentably $G$-symmetric monoidal stable $G$-$\infty$-category, a result of Hilman yields a $G$-symmetric monoidal structure on the $G$-functor categories 
\[\Fun(\myuline{I}^{\mathrm{op}},\myuline{\mathcal{C}})\]
extending the Day convolution monoidal structure on each of the functor categories $\Fun(\myuline{I}(G/H)^{\mathrm{op}},\myuline{\mathcal{C}}(G/H))$. In the following, let $\myuline{\mathrm{Cat}}$ denote the $G$-$\infty$-category of (small) $G$-$\infty$-category (denoted by $\myuline{\mathrm{Cat}}_{\mathcal{T}}^{{\underline{\mathrm{Idem}}}(\omega)}$ in \cite[Proposition 6.7.5]{kaifpresentable}).

\begin{proposition}\label{prop:kaiffunctor}
    There is a lax $G$-symmetric monoidal functor
\[\myuline{\mathrm{Cat}}\to \myuline{\mathrm{Pr}}_{L,st}^G\]
sending $\myuline{I}\mapsto\Fun(\myuline{I}^{\mathrm{op}},\myuline{\Sp})$, which is left adjoint to the forgetful functor. In particular there is an induced functor
\[\mathrm{CAlg}_G(\myuline{\mathrm{Cat}})\to\mathrm{CAlg}_G( \myuline{\mathrm{Pr}}_{L,st}^G),\]
equipping $\Fun(\myuline{I}^{\mathrm{op}},\myuline{\Sp})$ with a presentably $G$-symmetric monoidal structure, when $\myuline{I}$ has a $G$-symmetric monoidal structure.
\end{proposition}
\begin{proof}
    The claims are obtained by combining Propositions 2.3.9 and 2.3.12 in \cite{hilmanthesis}. In particular, the identification
    \[\Fun(\myuline{I}^{\mathrm{op}},\myuline{\mathcal{S}})\otimes\myuline{\Sp}\simeq \Fun(\myuline{I}^{\mathrm{op}},\myuline{\Sp})\]
    follows as in the nonequivariant case from the equivalences
    \[\myuline{\mathcal{C}}\otimes \myuline{\Sp}\simeq \Fun^R(\myuline{\mathcal{C}}^{\mathrm{op}},\myuline{\Sp}),\]
    of \cite[Proposition 2.2.23]{hilmanthesis}.
\end{proof}

\begin{theorem}\label{thm:gsymstructureonfil}
    There is a presentably $G$-symmetric monoidal $\infty$-category $\Fun(\myuline{I}^{\mathrm{op}},\myuline{\mathcal{C}})$ with the property that for each $H\subset G$, 
\[
\Fun(\myuline{I}^{\mathrm{op}},\myuline{\mathcal{C}})(G/H)=\Fun(\myuline{I}(G/H)^{\mathrm{op}},\myuline{\mathcal{C}}(G/H))
\]
    with its Day convolution symmetric monoidal structure. Moreover, the functors
    \begin{align*}
    \mathrm{colim}&: \Fun(\myuline{I}^{\mathrm{op}},\myuline{\mathcal{C}})\to\myuline{\mathcal{C}}\\
    Y&:\myuline{I}\to \Fun(\myuline{I}^{\mathrm{op}},\myuline{\mathcal{C}})
\end{align*}
have canonical $G$-symmetric monoidal structures. Here $\mathrm{colim}$ is given by applying the functor of the previous proposition to the $G$-symmetric monoidal functor $\myuline{I}\to\myuline{*}$, and the functor $Y$ is given by the unit of the adjunction of the previous proposition.
\end{theorem}
\begin{remark}
    Unlike in the non-parametrized setting, it is not known if the $G$-symmetric monoidal structure on $\Fun(\myuline{I}^{\mathrm{op}},\myuline{\mathcal{C}})$ given above coincides with the parametrized Day convolution structure of Nardin--Shah \cite{nardinshah}. This would follow if the Yoneda functor $Y:\myuline{I}\to \Fun(\myuline{I}^{\mathrm{op}},\myuline{\mathcal{C}})$ were $G$-symmetric monoidal with respect to Day convolution. We will not need this however, since the induced structure on 
    \[\Fun(\myuline{I}^{\mathrm{op}}(G/H),\myuline{\mathcal{C}}(G/H))\] coincides with the Day convolution structure, since in the non-parametrized case, the Yoneda functor is known to be symmetric monoidal with respect to Day convolution. We thank Emma Brink for explaining this to us.
\end{remark}

\begin{definition}\label{def:bigraded spheres}
    In the case $\myuline{I}=\myuline{\Z}$ and $\myuline{\mathcal{C}}=\myuline{\Sp}$, for $V\in RO(G)$, we define the bigraded sphere 
    \[\mathbb{S}^{V,s}:=\mathbb{S}^V\otimes Y(s+|V|)\in\Fun(\myuline{\Z}^{\mathrm{op}},\myuline{\Sp}),\]
    i.e. the pointwise tensor of the sphere $\mathbb{S}^V$ with the filtered object
    \[
    \cdots\to0\to \mathbb{S}^0=\mathbb{S}^0=\cdots
    \]
    starting at component $|V|+s$.
\end{definition}

\begin{lemma}\label{lemma:normsofbigradedspheres}
    The norm $N_H^G:\Fil(\Sp^H)\to\Fil(\Sp^G)$ in the $G$-symmetric monoidal $\infty$-category $\Fun(\myuline{\Z}^{\mathrm{op}},\myuline{\Sp})$ satisfies
    \[
    N_H^G(\mathbb{S}^{V,s})=\mathbb{S}^{\mathrm{Ind}_H^G(V),|G/H|s}
    \]
\end{lemma}
\begin{proof}
    Since $\Fun(\myuline{\Z}^{\mathrm{op}},\myuline{\Sp})$ is a presentably $G$-symmetric monoidal stable $\infty$-category, it is tensored over $\myuline{\Sp}$, and the norm functor $N_H^G$ commutes with the tensoring. By the usual formula $N_H^G(\mathbb{S}^V)=\mathbb{S}^{\mathrm{Ind}_H^GV}$ in $\myuline{\Sp}$, we therefore have
    \begin{align*}
    N_H^G(\mathbb{S}^{V,s})&=N_H^G(\mathbb{S}^V\otimes Y(s+|V|))\\&=N_H^G(\mathbb{S}^V)\otimes N_H^G(Y(s+|V|))\\
    &=\mathbb{S}^{\mathrm{Ind}_H^GV}\otimes Y(|G/H|(s+|V|)),
    \end{align*}
    using the fact that $Y$ is $G$-symmetric monoidal.
\end{proof}

When $\myuline{I}=\myuline{\Z}$, there is a canonical notion of taking the associated graded of an object $X\in\Fun(\myuline{\Z}^{\mathrm{op}},\myuline{\mathcal{C}})$, namely one has the functor
\begin{align*}
    \Fun(\myuline{\Z}^{\mathrm{op}},\myuline{\mathcal{C}})&\to \Fun((\myuline{\Z}^\delta)^{\mathrm{op}},\myuline{\mathcal{C}})\\
    X&\mapsto (\mathrm{cofib}(X^{n+1}\to X^n))_{n\in\Z}
\end{align*}
This functor also inherits a $G$-symmetric monoidal structure from $\myuline{\mathcal{C}}$.

\begin{theorem}
     There are presentably $G$-symmetric monoidal $\infty$-categories 
     \begin{align*}
        \Fil(\myuline{\mathcal{C}})&:=\Fun(\myuline{\Z}^{\mathrm{op}},\myuline{\mathcal{C}})\\
        \mathrm{gr}(\myuline{\mathcal{C}})&:=\Fun((\myuline{\Z}^\delta)^{\mathrm{op}},\myuline{\mathcal{C}})
     \end{align*}
     with the property that for each $H\subset G$, 
\begin{align*}
        \Fil(\myuline{\mathcal{C}})(G/H)&=\Fil(\myuline{\mathcal{C}}(G/H))\\
        \mathrm{gr}(\myuline{\mathcal{C}})&=\mathrm{gr}(\myuline{\mathcal{C}}(G/H))
     \end{align*}
    with their Day convolution symmetric monoidal structures. Moreover, the functors
    \begin{align*}
        \mathrm{colim}&:\Fil(\myuline{\mathcal{C}})\to\myuline{\mathcal{C}}\\
        \mathrm{gr}&:\Fil(\myuline{\mathcal{C}})\to\mathrm{gr}(\myuline{\mathcal{C}})\\
        \mathrm{forget}&:\mathrm{gr}(\myuline{\mathcal{C}})\to\myuline{\mathcal{C}}
    \end{align*}
     all admit the structure of $G$-symmetric monoidal left adjoints. Here $\mathrm{colim}$ and $\mathrm{forget}$ are given by the functor $\mathrm{colim}$ of the previous theorem in the cases $\myuline{I}=\myuline{\Z}$ and $\myuline{I}=\myuline{\Z}^\delta$, respectively. The functor $\mathrm{gr}$ is the associated graded functor discussed just before the theorem.
\end{theorem}
\begin{proof}
    We just need to prove that the functor $\mathrm{gr}$ admits a $G$-symmetric monoidal structure. It suffices prove this in the universal case $\myuline{\mathcal{C}}=\myuline{\Sp}$ via the equivalence
    \[
    \Fun(\underline{I}^{\mathrm{op}},\mathcal{\underline{C}})\simeq \Fun(\underline{I}^{\mathrm{op}},\myuline{\Sp})\otimes\mathcal{\underline{C}},
    \]
    as in Proposition \ref{prop:kaiffunctor}. This is proven in the same way as \cite[Proposition 3.2.1]{rotation} with small modifications, so we will briefly give the details of adapting this to the $G$-equivariant setting. The functor $\myuline{\Z}^\delta\to \myuline{\Z}$ that is the identity on objects is $G$-symmetric monoidal, so by the proposition it induces a $G$-symmetric monoidal left adjoint $L:\mathrm{gr}(\myuline{\Sp})\to \Fil(\myuline{\Sp})$ with $G$-lax-symmetric monoidal right adjoint $\mathrm{Res}:\Fil(\myuline{\Sp})\to \mathrm{gr}(\myuline{\Sp})$, and we let $\mathbb{S}[t]$ denote the $G$-commutative algebra  $\mathrm{Res}(\mathbbm{1})$, with notation as in \cite[Remark 3.1.3,Notation 3.1.4]{rotation}. 
    
    We follow the proof of \cite[Proposition 3.1.6]{rotation} to show that the functor
    \[F:\Fil(\myuline{\Sp})=\mathrm
    {Mod}_{{\Fil({\myuline{\scriptstyle{\Sp}}})}}(\mathbbm{1})\to \mathrm{Mod}_{\mathrm{gr}({\myuline{\scriptstyle{\Sp}}})}(\mathbb{S}[t])\]
    is a $G$-symmetric monoidal equivalence of categories. First, Lurie shows this is an equivalence of categories, which goes through without change by replacing each shifted sphere $\mathbb{S}^{0,n}$ (denoted $\mathbb{S}(n)$ by Lurie) with all induced shifted spheres $\mathrm{Ind}_H^G\mathbb{S}^{0,n}$. It suffices then to show that the $G$-lax symmetric monoidal structure on $F$ is $G$-symmetric monoidal. Lurie's argument goes through again replacing each shifted sphere $\mathbb{S}^{0,n}$ with all induced shifted spheres $\mathrm{Ind}_H^G\mathbb{S}^{0,n}$ to show that
\[F(X)\otimes_{\mathbb{S}[t]}F(Y)\to F(X\otimes Y)\]
    is an equivalence for all $X,Y$. This reduces us to showing that $F$ preserves norms, that is, for any $X\in \Fil(\Sp^H)$, and $H\subset K\subset G$, the map
    \begin{equation}\label{eq:normcomparison}
    N_H^KF(X)\to F(N_H^KX)
    \end{equation}
    is an equivalence. By induction on the order of the group $G$, we may reduce to the case when $K=G$. We may check that this map is an equivalence after applying the geometric fixed point functors $\Phi^K$ for each $K$, and again by induction on $|G|$, we may reduce to the case $K=G$. By use of Proposition \ref{prop:geomfixedpointsofnorm}, we see that both sides of (\ref{eq:normcomparison}) commute with colimits in $X$ after applying $\Phi^K$, so we may as before reduce to the case when $X$ is an (induced) shifted sphere. This follows as in the nonequivariant case by a direct check using Lemma \ref{lemma:normsofbigradedspheres}.
    
    The rest of the proof follows exactly as in the nonequivariant case by repeating the proofs of \cite[Propositions 3.2.3,3.2.5,3.2.7]{rotation} in our setting.
\end{proof}

In the general case, there are many ways one might define an associated graded functor $\Fun(\myuline{I}^{\mathrm{op}},\myuline{\mathcal{C}})\to \Fun((\myuline{I}^\delta)^{\mathrm{op}},\myuline{\mathcal{C}})$. In the $\myuline{RO}$ case, one can opt for a conservative functor to $\myuline{\Z}$-filtered objects that resembles associating the total complex to a double complex (or more generally a multicomplex). After passing to $\myuline{\Z}$-filtered objects, one may apply the associated graded functor on $\Fil(\myuline{\Sp})$, and the composite $\Fun(\myuline{RO}^{\mathrm{op}},\myuline{\Sp})\to\mathrm{gr}(\myuline{\Sp})$ is $G$-symmetric monoidal.

\begin{proposition}\label{prop:total}
    Let $\mathrm{dim}:\myuline{RO}\to\myuline{\Z}$ be the morphism of $G$-symmetric monoidal posets sending a representation $V$ to its dimension $|V|$. Restriction along $\mathrm{dim}$ induces a $G$-lax symmetric monoidal functor
    \[\mathrm{dim}^*:\Fil(\myuline{\mathcal{C}})\to\Fun(\myuline{RO}^{\mathrm{op}},\myuline{\mathcal{C}})\]
    with $G$-symmetric monoidal left adjoint
    \[\mathrm{Total}:\Fun(\myuline{RO}^{\mathrm{op}},\myuline{\mathcal{C}})\to\Fil(\myuline{\mathcal{C}})\]
    where 
    \[\mathrm{Total}^n(X)(G/H)=\bigoplus\limits_{\substack{V\in RO(H)\\|V|=n}}X^V(G/H)\]
\end{proposition}
\begin{proof}
The $G$-symmetric monoidal left adjoint $\mathrm{Total}$ is obtained by applying the functor of Proposition \ref{prop:kaiffunctor} to $\mathrm{dim}:\myuline{RO}\to\myuline{\Z}$, which also gives the lax-$G$-structure on its right adjoint, the restriction along $\mathrm{dim}$. 

To compute $\mathrm{Total}^n$, note that for any $G$, there is an adjunction
\[
\begin{tikzcd}
    \Fil(\Sp^G)\arrow[d,"\mathrm{ev}_n"]\\
    \Sp^G\arrow[u,shift left=4ex,"-\otimes\mathbb{S}^{0,n}"]\arrow[u,shift right=4ex,"-\otimes\widehat{\mathbb{S}(n)}"']
\end{tikzcd}
\]
where $\mathrm{ev}_n$ denotes the evaluation of a filtered $G$-spectrum $X$ at the $n$-th component $X^n$, and $\widehat{\mathbb{S}(n)}$ is the filtered $G$-spectrum
\[\cdots =\mathbb{S}=\mathbb{S}\to0\to\cdots\]
with 
\[(\widehat{\mathbb{S}(n)})^i\begin{cases}
    \mathbb{S}&i\ge n\\0&i<n
\end{cases}\]
Indeed, the description of the lefthand adjunction follows by the Yoneda lemma since $\mathbb{S}^{0,n}=Y(n)$. For the righthand adjunction, since colimits are computed pointwise in $\Fil(\Sp^G)$ it follows that $\mathrm{ev}_n$ preserves compact objects, so its right adjoint preserves colimits. The description of its right adjoint may then be checked on the induced spheres $\mathrm
{Ind}_H^G\mathbb{S}$, where it is seen directly.

By composition of adjoints it follows that the functor $\mathrm{ev}_n\circ \mathrm{Total}$ is left adjoint to the composite $\mathrm{dim}^*\circ (-\otimes\widehat{\mathbb{S}(n)})$, and one checks directly that, for $Z\in\Sp^G$,
\[\mathrm{dim}^*(Z\otimes\widehat{\mathbb{S}(n)})^V=\begin{cases}
    Z&|V|\ge n\\0&|V|<n
\end{cases}\]
By the same argument used to compute the right adjoint of $\mathrm{ev}_n$, one sees that the left adjoint to $\mathrm{dim}^*\circ (-\otimes\widehat{\mathbb{S}(n)})$ is given by 
\[\bigoplus\limits_{\substack{V\in RO(G)\\|V|=n}}\mathrm{ev}_V\qedhere\]
\end{proof}

\begin{remark}
    The analogue of Proposition \ref{prop:total} holds with $RU$ in place of $RO$ and the complex dimension function in place of $\mathrm{dim}$.
\end{remark}

\begin{remark}\label{rmk:Totalisafiltrationofwhat}
    For $X\in \Fun(\myuline{RO},\myuline{\Sp})$, the filtered object $\mathrm{Total}(X)$ is a filtration of $\mathrm{colim} (X)$ because $\mathrm{dim}^*\circ\mathrm{const}=\mathrm{const}$ and $\tau^{-1}\circ\mathrm{Total}$ is left adjoint to $\mathrm{dim}^*\circ\mathrm{const}$.
\end{remark}

\subsection{Geometric fixed points and dilation} There is an equivalence of functors $\Phi^G\circ N_H^G\simeq\Phi^H:\Sp^H\to \Sp$, and as we saw in Proposition \ref{prop:geomfixedpointsofnorm}, for a general $G$-symmetric monoidal stable $\infty$-category $\myuline{\mathcal{C}}$, we at least know that  $\Phi^G\circ N_H^G$ preserves colimits. In the filtered case, the classical formula $\Phi^G\circ N_H^G\simeq\Phi^H$ holds up to \emph{dilation}. To define this, we first need to recall a universal property of filtered objects.

\begin{proposition}\label{prop:universalpropertyfiltered}
    Let $I$ be a symmetric monoidal poset and $\mathcal{C}$ a presentably symmetric monoidal stable $\infty$-category. There is a symmetric monoidal equivalence
    \[\Fun(I^{\mathrm{op}},\mathcal{C})\simeq \mathcal{C}\otimes \Fun(I^{\mathrm{op}},\Sp).\]
    Moreover, for any presentably symmetric monoidal stable category $\mathcal{D}$, the functor $Y$ induces a natural equivalence
    \[\mathrm{Fun}^{\otimes}(I,\mathcal{D})\times \mathrm{LFun}^{\otimes}(\mathcal{C},\mathcal{D})\simeq\mathrm{LFun}^{\otimes}(\Fun(I^{\mathrm{op}},\mathcal{C}),\mathcal{D}),\]
    where $\mathrm{Fun}^{\otimes}$ denotes symmetric monoidal functors, and $\mathrm{LFun}^{\otimes}$ denotes symmetric monoidal colimit preserving functors.
\end{proposition}
\begin{proof}
    The first statement follows as in Proposition \ref{prop:kaiffunctor} from the formula $\mathcal{C}\otimes\mathcal{D}\simeq\mathrm{Fun}^R(\mathcal{C}^\mathrm{op},\mathcal{D})$ for the tensor product in $\mathrm{Pr}^L$. Since $\otimes$ is the coproduct in $\mathrm{CAlg}(\mathrm{Pr}^L_{st})$, the first statement allows us to reduce the second statement to the case $\mathcal{C}=\Sp$, which states
    that
    \[
    \mathrm{LFun}^{\otimes}(\Fun(I^{\mathrm{op}},\Sp),\mathcal{D})\simeq \Fun^\otimes(I,\mathcal{D}).
    \]
    Note that there is a symmetric monoidal equivalence $\Fun(I^{\mathrm{op}},\Sp)\simeq\Sp(\mathrm{Psh}(I))$. Using that stabilization is a smashing localization in $\mathrm{Pr}^L$, we see that
    \[
    \mathrm{LFun}^{\otimes}(\Fun(I^{\mathrm{op}},\Sp),\mathcal{D})\simeq \mathrm{LFun}^{\otimes}(\mathrm{Psh}(I),\mathcal{D}),
    \]
    and now the claim follows by universal property of Day convolution on presheaf categories.
\end{proof}

\begin{definition}\label{def:dilation}
    Let $\mathcal{C}$ be a presentably symmetric monoidal stable $\infty$-category, and let $p:I\to J$ be a morphism of symmetric monoidal posets. We define the \emph{dilation operator}
    \[\mathcal{D}^p:\Fun(I^{\mathrm{op}},\mathcal{C})\to \Fun(J^{\mathrm{op}},\mathcal{C})\]
 to be the symmetric monoidal colimit preserving functor  defined by the symmetric monoidal functor
\[I\xrightarrow{p}J\xrightarrow{Y}\Fun(J^{\mathrm{op}},\Sp)\to\Fun(J^{\mathrm{op}},\mathcal{C}),\]
where the last functor is postcomposition with the unique symmetric monoidal colimit preserving functor $\Sp\to \mathcal{C}$, and 
\[f_\mathcal{C}:\mathcal{C}\simeq\mathcal{C}\otimes\Sp\xrightarrow{\mathrm{id}\otimes f}\mathcal{C}\otimes\Fun(J^{\mathrm{op}},\Sp)\simeq\Fun(J^{\mathrm{op}},\mathcal{C}),\]
under the equivalences of Proposition \ref{prop:universalpropertyfiltered}, where $f$ is the unique symmetric monoidal colimit preserving functor
\[\Sp\to \Fun(J^{\mathrm{op}},\mathcal{C}).\]
\end{definition}

\begin{example}\label{ex:dilationexamples}
    Consider the multiplication map $n:\Z\to \Z$. This induces a dilation operator $\mathcal{D}^n:\Fil(\mathcal{C})\to \Fil(\mathcal{C})$. This may be equivalently defined as satisfying $(\mathcal{D}^nX)^m=X^{\lceil\frac{m}{n}\rceil}$, as may be checked using Proposition \ref{prop:universalpropertyfiltered}.
    
        This construction appears in the case $n=2$ in the work of Meier--Shi--Zeng \cite[Definition 3.5]{meiershizengtranschromatic}, where it is called the doubling operator and denoted $\mathcal{D}$.
\end{example}


In the following, we will use the notation $\Phi^G$ for the geometric fixed points functor of Definition \ref{def:geometricfixedpointsfunctor} on $\Fun(\myuline{I}^{\mathrm{op}},\myuline{\Sp})$. By Lemma \ref{lemma:geometricfixedpointsatG/G}, this coincides at the $G/G$ level with the functor given by applying $\Phi^G:\Sp^G\to \Sp$ pointwise. Moreover, for $\underline{I}$ a $G$-symmetric monoidal poset, we let $\mathcal{D}_H^G$ be the dilation operator induced by the norm $\underline{I}(G/H)\to \underline{I}(G/G)$. 

\begin{corollary}\label{cor:normsandgeometricfixedpoints}
 There is an equivalence
    \[\Phi^G\circ N_H^G\simeq \mathcal{D}_H^G\circ\Phi^H\]
    of symmetric-monoidal colimit preserving functors 
    \[\Fun(\myuline{I}(G/H)^{\mathrm{op}},\Sp^H)\to\Fun(\myuline{I}(G/G)^{\mathrm{op}},\Sp).\]
\end{corollary}
\begin{proof}
    Note that by Proposition \ref{prop:geomfixedpointsofnorm}, the lefthand side is colimit preserving. By Proposition \ref{prop:universalpropertyfiltered}, it therefore suffices to produce an equivalence of symmetric monoidal functors as in the statement of the corollary after precomposing with $f_{\Sp^H}:\Sp^H\to \Fil(\Sp^H)$ and $Y_H:\myuline{I}(G/H)\to \Fil(\Sp^H)$, respectively. The former follows from the fact that $f_{\Sp^H}:\Sp^H\to \Fun(\myuline{I}(G/H)^{\mathrm{op}},\Sp^H)$ refines to the unique $G$-symmetric monoidal left adjoint $\myuline{\Sp}\to \Fun(I^{\mathrm{op}},\Sp)$, so that 
\begin{align*}
    \Phi^G\circ N_H^G\circ f_{\Sp^H}&\simeq \Phi^G\circ f_{\Sp^G}\circ N_H^G\\
    &\simeq f_{\Sp}\circ \Phi^G\circ N_H^G\\
    &\simeq f_{\Sp}\circ \Phi^H\\
    &\simeq \Phi^H\circ f_{\Sp^H}\\&\simeq \mathcal{D}_H^G\circ\Phi^H\circ f_{\Sp^H},
\end{align*}
where we use the equivalence $\Phi^G\circ N_H^G\simeq \Phi^H$ of functors $\Sp^H\to \Sp$.

The latter case follows similarly from the fact that 
\[Y_H:\myuline{I}(G/H)\to \Fun(\myuline{I}(G/H)^{\mathrm{op}},\Sp^H)\]
refines to the $G$-symmetric monoidal functor $Y:\myuline{I}\to\Fun(\underline{I}^{\mathrm{op}},\Sp)$. Indeed, this implies that
\begin{align*}
    \Phi^G\circ N_H^G\circ Y_H&\simeq \Phi^G\circ Y_G\circ N_H^G\simeq Y\circ N_H^G,
\end{align*}
where the latter $N_H^G$ is the norm $\myuline{I}(G/H)\to \myuline{I}(G/G)$ in $\myuline{I}$. One must therefore produce an equivalence of symmetric monoidal functors between 
\[\myuline{I}(G/H)\xrightarrow{N_H^G}\myuline{I}(G/G)\xrightarrow{Y} \Fun(\myuline{I}(G/G)^{\mathrm{op}},\Sp)\]
and
\[\myuline{I}(G/H)\xrightarrow{Y} \Fun(\myuline{I}(G/H)^{\mathrm{op}},\Sp)\xrightarrow{\mathcal{D}_H^G}\Fun(\myuline{I}(G/G)^{\mathrm{op}},\Sp),\]
which is how $\mathcal{D}_H^G$ is defined.
\end{proof}

\section{The $RO(G)$-graded Whitehead filtration and equivariant $t$-structures}\label{sec:equivarianttstructures} In this section, we revisit the notion of a slice filtration on a category in the context of $t$-structures on filtered objects. We set these up also in a $G$-$\infty$-categorical context so as to obtain $G$-symmetric monoidal structures on slice connective covers.

\subsection{Filtered $t$-structures and slice filtrations} Fix $I$ a symmetric monoidal poset and $\mathcal{C}$ a presentably symmetric monoidal stable $\infty$-category. By \cite[Proposition 1.4.4.11]{HA}, given a small collection of objects $\{g_\alpha\}$ in $\mathcal{C}$, there is a $t$-structure on $\mathcal{C}$ whose connective objects consist of the closure $\mathcal{C}'$ of $\{g_\alpha\}$ under colimits and extensions. When we speak of $t$-structures on $\mathcal{C}$, we will always mean a $t$-structure that arises in this way, and we will refer to a $t$-structure both by its corresponding category of connective objects $\mathcal{C}'$ and by its connective cover functor $\tau_{\ge0}$. 

Following the notation of the previous section, we let $Y$ denote the functor
    \[Y:I\to \mathrm{Fun}(I^{\mathrm{op}},\Sp)\to \mathrm{Fun}(I^{\mathrm{op}},\mathcal{C}),\]
    where the first functor is ($\Sigma^\infty_+$ of) the Yoneda embedding and the latter is postcomposition with the unique symmetric monoidal left adjoint $\Sp\to\mathcal{C}$.

\begin{proposition}\label{prop:slicefiltration}
    For each $i\in I$, fix a $t$-structure $\mathcal{C}_i$ on $\mathcal{C}$ generated by a small set of objects $\{g_{\alpha,i}\}$ with connective cover functor $\tau_{\ge0}^i$. We have the following
    \begin{enumerate}
        \item There is a $t$-structure on $\mathrm{Fun}(I^\mathrm{op},\mathcal{C})$ such that $X\in\mathrm{Fun}(I^\mathrm{op},\mathcal{C})$ is connective if and only if $X^i\in\mathcal{C}_i$ for all $i\in I$. This $t$-structure is generated by the set \[\{g_{\alpha,i}\otimes Y(i)\}_{i,\alpha}\]
        ranging over all $\alpha$ and $i$.
    \end{enumerate}
    Suppose now that whenever $i\ge j$, one has an inclusion $\mathcal{C}_i\subset\mathcal{C}_j$, and let $\mathrm{Fun}(I^\mathrm{op},\mathcal{C})_{\ge0}$ and $\tau_{\ge0}$ denote the corresponding connective objects and connective cover functor on $\mathrm{Fun}(I^\mathrm{op},\mathcal{C})$, respectively. Then one has the following:
    \begin{enumerate}
    \setcounter{enumi}{2}
        \item  The functor $\tau_{\ge0}$ satisfies the formula
    \[
    (\tau_{\ge0}X)^i=\tau_{\ge0}^i(X^i)
    \]
    for $X\in \mathrm{Fun}(I^\mathrm{op},\mathcal{C})$, for all $i\in I$.   
    \item Suppose the unit $\mathbbm{1}\in\mathcal{C}_0$ and that $g_{\alpha,i}\otimes g_{\beta,j}\in \mathcal{C}_{i\otimes j}$, for all $i,j,\alpha,\beta$.
    Then the $t$-structure $\tau_{\ge0}$ is compatible with the monoidal structure in the sense that the inclusion $\mathrm{Fun}(I^\mathrm{op},\mathcal{C})_{\ge0}\hookrightarrow\mathrm{Fun}(I^\mathrm{op},\mathcal{C})$ is a symmetric monoidal left adjoint with respect to Day convolution.
     \item If, for all $i\in I$, the $t$-structure generated by $\mathcal{C}_i$ is (left/right) complete, then the $t$-structure on $\Fun(I^{\mathrm{op}},\mathcal{C})$ is (left/right) complete.
    \end{enumerate}

\end{proposition}
\begin{proof}
    Item (1) of the proposition follows directly from \cite[Proposition 1.4.4.11]{HA}. For (2), let $\mathrm{ev}_i:\mathrm{Fun}(I^{\mathrm{op}},\mathcal{C})\to \mathcal{C}$ denote the evaluation functor at $i\in I$. We determine the functor
    \[\mathrm{Fun}(I^{\mathrm{op}},\mathcal{C})\xrightarrow{\tau_{\ge0}}\mathrm{Fun}(I^{\mathrm{op}},\mathcal{C})_{\ge0}\xrightarrow{\mathrm{ev}_i}\mathcal{C}\]
    by computing its left adjoint. First, note that since colimits and extensions in $\mathrm{Fun}(I^{\mathrm{op}},\mathcal{C})$ are determined pointwise, one has a factorization
    \[
    \begin{tikzcd}
        \mathrm{Fun}(I^{\mathrm{op}},\mathcal{C})_{\ge0}\arrow[r]\arrow[d]&\mathcal{C}_i\arrow[d]\\
        \mathrm{Fun}(I^{\mathrm{op}},\mathcal{C})\arrow[r,"\mathrm{ev}_i"]&\mathcal{C}
    \end{tikzcd}
    \]
    where the vertical arrows are the inclusions. Conversely, the left adjoint $L_i$ to the functor $\mathrm{ev}_i$ is given by the composite
    
\begin{align*}
L_i:\mathcal{C}\simeq\{i\}\times\mathcal{C}\hookrightarrow I\times\mathcal{C}&\xrightarrow{Y\times\mathrm{id}}\mathrm{Fun}(I^{\mathrm{op}},\mathcal{C})\times\mathcal{C}\\
&\xrightarrow{\mathrm{id}\times\mathrm{const}}\mathrm{Fun}(I^{\mathrm{op}},\mathcal{C})\times \mathrm{Fun}(I^{\mathrm{op}},\mathcal{C})\\
    &\xrightarrow{\otimes}\mathrm{Fun}(I^{\mathrm{op}},\mathcal{C})
\end{align*}
and it follows directly from the assumption that $i\ge j\implies \mathcal{C}_i\subset\mathcal{C}_j$ that one has a factorization
\[
\begin{tikzcd}
    \mathcal{C}_i\arrow[r,"L_i|_{\mathcal{C}_i}"]\arrow[d]&\mathrm{Fun}(I^{\mathrm{op}},\mathcal{C})_{\ge0}\arrow[d]\\
    \mathcal{C}\arrow[r,"L_i"]&\mathrm{Fun}(I^{\mathrm{op}},\mathcal{C})
\end{tikzcd}
\]    
    where the vertical arrows are the inclusions. It follows that $L_i|_{\mathcal{C}_i}$ is left adjoint to $\mathrm{Fun}(I^{\mathrm{op}},\mathcal{C})_{\ge0}\xrightarrow{\mathrm{ev}_i}\mathcal{C}_i$, and that the composite
    \[\mathcal{C}_i\hookrightarrow\mathcal{C}\xrightarrow{L_i}\mathrm{Fun}(I^{\mathrm{op}},\mathcal{C})\]
    is left adjoint to 
    \[\mathrm{Fun}(I^{\mathrm{op}},\mathcal{C})\xrightarrow{\tau_{\ge0}}\mathrm{Fun}(I^{\mathrm{op}},\mathcal{C})_{\ge0}\xrightarrow{\mathrm{ev}_i}\mathcal{C}_i,\]
    which produces the desired formula.
    
    For (3), we need to check that $\mathrm{Fun}(I^{\mathrm{op}},\mathcal{C})_{\ge0}$ contains the unit $Y(\mathbbm{1})$ and is closed under tensor products. The former follows from the fact that $\mathbbm{1}\in\mathcal{C}_0$ and the key assumption that $i\ge j\implies \mathcal{C}_i\subset\mathcal{C}_j$. For the latter, we must check that, if $X^i,Y^i\in\mathcal{C}_i$ for all $i\in I$, then
    \[(X\otimes Y)^i=\mathrm{colim}_{j\otimes k\ge i}(X^j\otimes Y^k)\in\mathcal{C}_i.\]
    This follows by the assumption on generators and the fact that $\mathcal{C}_i$ is closed under colimits. Finally, (4) follows from the connective cover formula of (2) and \cite[1.2.1.19]{HA}.
    \end{proof}

\begin{definition}\label{def:slicefiltration}
We call an $I$-indexed sequence of $t$-structures $\mathcal{C}_i$ satisfying
\[i\ge j\implies\mathcal{C}_i\subset\mathcal{C}_j\]
and item (4) of Proposition \ref{prop:slicefiltration} an $I$-indexed \emph{slice filtration} on $\mathcal{C}$. We will use this term to refer to both this data and to the $t$-structure on $\Fun(I^{\mathrm{op}},\mathcal{C})$ given by Proposition \ref{prop:slicefiltration}.   
\end{definition}

\begin{example}\label{example:slicefiltration}
    Let $\mathcal{C}=\Sp^G$ and $I=\Z$ with its standard poset structure. For each $n\in\Z$ the set $\{\mathrm{Ind}_H^G\mathbb{S}^{m\rho_H}:m|H|\ge n\}$ generates a $t$-structure $\mathcal{C}'=\Sp^G_{\ge n}$ on $\Sp^G$ with the property that $\Sp^G_{\ge n+1}\subset\Sp^G_{\ge n}$ for all $n$. This gives a slice filtration on $\Sp^G$ in the above sense, and it recovers the usual equivariant slice filtration, as the composite
    \[\Sp^G\xrightarrow{\mathrm{const}}\Fil(\Sp^G)\xrightarrow{\tau_{\ge0}}\Fil(\Sp^G)\]
    sends a $G$-spectrum to its slice tower. We return to this example in Definition \ref{def:twofilteredtstructures}
\end{example}

We will need also a $G$-$\infty$-categorical version of Proposition \ref{prop:slicefiltration} for the purposes of obtaining $G$-lax symmetric monoidal structures on the connective cover functors $\tau_{\ge0}:\mathrm{Fun}(I^\mathrm{op},\mathcal{C})\to \mathrm{Fun}(I^\mathrm{op},\mathcal{C})_{\ge0}$ from Proposition \ref{prop:slicefiltration}.

\begin{proposition}\label{prop:gmonoidalslicefiltrations}
   Fix $\myuline{I}$ a $G$-symmetric monoidal poset and $\myuline{\mathcal{C}}$ a presentably $G$-symmetric monoidal stable $\infty$-category, and for each $H\subset G$, fix an $\myuline{I}(G/H)$-indexed slice filtration on $\myuline{\mathcal{C}}(G/H)$. For $H\subset G$, let $\{g_{\alpha,i}^H\otimes Y(i)\}_{i,\alpha}$ be the generators of $\Fun(\myuline{I}(G/H)^{\mathrm{op}},\myuline{\mathcal{C}}(G/H))_{\ge0}$ as in Proposition \ref{prop:slicefiltration}. Suppose that for any  subgroups $H,H'\subset G$ and any map of $G$-sets $f:G/H\to G/H'$, one has inclusions
\[f^*(\{g_{\alpha,i}^{H'}\otimes Y(i)\}_{i,\alpha})\subset \underline{\mathcal{C}}(G/H)_{f^*(i)}\]
   for all $i\in\myuline{I}(G/H')$ and that for any $H\subset H'\subset G$, one has inclusions
   \[
   N_{H}^{H'}(\{g_{\alpha,i}^H\otimes Y(i)\}_{i,\alpha})\subset \myuline{\mathcal{C}}(G/H')_{N_H^{H'}(i)}
   \]
   for all $i\in\myuline{I}(G/H)$. Then the assignment 
   \[\Fun(\myuline{I}^{\mathrm{op}},\myuline{\mathcal{C}})_{\ge0}(G/H):=\Fun(\myuline{I}(G/H)^{\mathrm{op}},\myuline{\mathcal{C}}(G/H))_{\ge0}\]
   defines a $G$-symmetric monoidal subcategory of $\Fun(\myuline{I}^{\mathrm{op}},\myuline{\mathcal{C}})$, and 
   the connective cover functors of Proposition \ref{prop:slicefiltration} on $\Fun(\myuline{I}(G/H)^{\mathrm{op}},\myuline{\mathcal{C}}(G/H))$ assemble into a lax $G$-symmetric monoidal right adjoint
   \[\tau_{\ge0}:\Fun(\myuline{I}^{\mathrm{op}},\myuline{\mathcal{C}})\to \Fun(\myuline{I}^{\mathrm{op}},\myuline{\mathcal{C}})_{\ge0}\]
\end{proposition}
\begin{proof}
    The assumption that
    \[f^*(\myuline{\mathcal{C}}(G/H')_i)\subset \underline{\mathcal{C}}(G/H)_{f^*(i)}\]
    for all $f:G/H\to G/H'$ guarantees that $\Fun(\myuline{I}^{\mathrm{op}},\myuline{\mathcal{C}})_{\ge0}$ is a sub $G$-category of $\Fun(\myuline{I}^{\mathrm{op}},\myuline{\mathcal{C}})$ and that the inclusion 
    \[\Fun(\myuline{I}^{\mathrm{op}},\myuline{\mathcal{C}})_{\ge0}\to \Fun(\myuline{I}^{\mathrm{op}},\myuline{\mathcal{C}})\]
    is a $G$-left adjoint by applying \cite[Lemma 5.12]{CHLL}. It remains to show that $\Fun(\myuline{I}^{\mathrm{op}},\myuline{\mathcal{C}})_{\ge0}$ is closed under the $G$-symmetric monoidal structure. This is guaranteed levelwise by Proposition \ref{prop:slicefiltration} (3) since we have fixed a slice filtration levelwise, so it remains to show that for each $H\subset H'\subset G$, there is an inclusion
    \[N_H^{H'}(\Fun(\myuline{I}(G/H)^{\mathrm{op}},\myuline{\mathcal{C}}(G/H))_{\ge0})\subset \Fun(\myuline{I}(G/H')^{\mathrm{op}},\myuline{\mathcal{C}}(G/H'))_{\ge0}.\]
    By induction on the order of the group, we may reduce to the case $H'=G$. By assumption, we have that
\[N_H^G(g_{\alpha,i}^H\otimes Y(i))\in \Fun(\myuline{I}(G/G)^{\mathrm{op}},\myuline{\mathcal{C}}(G/G))_{\ge0}\]
   for all $\alpha,i$. For a general $X\in \Fun(\myuline{I}(G/H)^{\mathrm{op}},\myuline{\mathcal{C}}(G/H))_{\ge0}$, consider the cofiber sequence
   \[E\mathcal{P}_+\otimes N_H^G(X)\to N_H^G(X)\to \tilde{E}\mathcal{P}\otimes N_H^G(X).\]
   Since the functor $\tilde{E}\mathcal{P}\otimes N_H^G(-)$ commutes with colimits by Proposition \ref{prop:geomfixedpointsofnorm}, it follows that $\tilde{E}\mathcal{P}\otimes N_H^G(X) \in \Fun(\myuline{I}(G/G)^{\mathrm{op}},\myuline{\mathcal{C}}(G/G))_{\ge0}$, so it remains to show that $E\mathcal{P}_+\otimes N_H^G(X)\in \Fun(\myuline{I}(G/G)^{\mathrm{op}},\myuline{\mathcal{C}}(G/G))_{\ge0}$. Since $E\mathcal{P}_+$ has a filtration with associated graded consisting of sums of $G$-spectra of the form $\mathrm{Ind}_K^G(\mathbb{S}^n)$ for $n\ge0$ and $K$ a proper subgroup, this case follows by induction. 
\end{proof}

\begin{definition}\label{def:normedslicefiltration}
For $\myuline{I}$ and $\underline{\mathcal{C}}$ as above, we call a collection of $\myuline{I}(G/H)$-indexed slice filtrations on $\myuline{\mathcal{C}}(G/H)$ satisfying the conditions of Proposition \ref{prop:gmonoidalslicefiltrations} a \emph{normed $\myuline{I}$-indexed slice filtration} on $\myuline{\mathcal{C}}$.    
\end{definition}

\subsection{The $RO(G)$-graded Whitehead tower and the equivariant AN filtration}\label{subsec:ROwhitehead} We begin by recalling the definition of the Whitehead tower functor. Let $\mathcal{C}_i\subset \Sp^G$ be those $G$-spectra $X$ with $\myuline{\pi}_kX=0$ for $k<i$, i.e. $\mathcal{C}_i=\langle\mathrm{Ind}_H^G\mathbb{S}^i\rangle_{H\subset G}$. The $\mathcal{C}_i$ together form a $\Z$-indexed slice filtration on $\Sp^G$ in the sense of Definition \ref{def:slicefiltration}.

\begin{definition}\label{def:whitehead}
    The Whitehead tower functor is the composite
    \[\mathrm{Wh}:\Sp^G\xrightarrow{\mathrm{const}}\Fun(\Z^{\mathrm{op}},\Sp^G)\xrightarrow{\tau_{\ge0}^{\{\mathcal{C}_i\}}}\Fun(\Z^{\mathrm{op}},\Sp^G)_{\ge0}\]
\end{definition}

This functor has a lax symmetric monoidal structure since the $\{\mathcal{C}_i\}$ define a slice filtration. However, this structure does not extend to that of a normed slice filtration in the sense of Definition \ref{def:normedslicefiltration}. For example, while $\mathbb{S}^1\in \Sp_{\ge 1}$, the $C_2$-spectrum $\mathbb{S}^\rho=N_e^{C_2}(S^1)\notin\Sp^{C_2}_{\ge2}$ as one has the nonzero Euler class $a_\sigma\in \pi_1\mathbb{S}^\rho$. As a consequence, $\mathrm{Wh}(-)$ does not admit a lax $G$-symmetric monoidal structure.

There are at least two ``corrections'' to this issue. One is to use the Hill--Hopkins--Ravenel slice tower in place of the Whitehead tower, as in Example \ref{example:slicefiltration}. Another is to replace the indexing category $\myuline{\Z}$ with $\myuline{RO}$.  For each $H\subset G$ and $V\in RO(H)$, define 
    \[\mathcal{C}_V=\{X\in\Sp^H:\Sigma^{-V}X\text{ is a connective $H$-spectrum}\}.\]
We thank Lucas Piessevaux for explaining the following construction to us.

\begin{proposition}
    There is a normed $\myuline{RO}$-indexed slice filtration on $\myuline{\Sp}$ in the sense of Definition \ref{def:normedslicefiltration} with connective objects at the $V$-th component given by $\mathcal{C}_V\subset \Sp^H$, for $V\in RO(H)$. The connective cover functor corresponding to $\mathcal{C}_V$ is given by
    \[\tau_{\ge0}^V(X)=\Sigma^V\tau_{\ge0}(\Sigma^{-V}X)\]
    for $X\in\Sp^H$, where $\tau_{\ge0}$ is the standard connective cover functor in $\Sp^G$.
\end{proposition}
\begin{proof}
    One must check the conditions of Proposition \ref{prop:gmonoidalslicefiltrations}, and this is straightforward using the fact that each $\mathcal{C}_V$ is generated by
    \[\{\Sigma^V\mathrm{Ind}_K^H\mathbb{S}\}_{K\subset H}\]
    and the distributive law  of \cite[Proposition A.37]{HHR} expressing norms of indexed coproducts in terms of indexed coproducts of norms (cf. \cite[Proposition 11.1.10]{HHRbook}).
    The description of $\tau_{\ge0}^V$ follows by composing adjunctions and the equivalence 
    \[\begin{tikzcd}
        \Sigma^V:\mathcal{C}_0\arrow[r,shift left]&\mathcal{C}_V:\Sigma^{-V}\arrow[l,shift left]
    \end{tikzcd}.\qedhere\] 
\end{proof}

\begin{definition}[Piessevaux]\label{def:ROwhitehead}
    The $\myuline{RO}$-indexed Whitehead tower functor is the composite
    \[\mathrm{Wh}^{RO}:\myuline{\Sp}\xrightarrow{\mathrm{const}}\Fun(\myuline{RO}^{\mathrm{op}},\myuline{\Sp})\xrightarrow{\tau_{\ge0}^{\{\mathcal{C}_V\}}}\Fun(\myuline{RO}^{\mathrm{op}},\myuline{\Sp})_{\ge0}\]
    $\mathrm{Wh}^{RO}$ has a canonical lax $G$-symmetric monoidal structure as $\mathrm{const}$ is the right adjoint to the $G$-symmetric monoidal functor $\mathrm{colim}:\Fun(\myuline{RO}^{\mathrm{op}},\myuline{\Sp})\to \myuline{\Sp}$ of Theorem \ref{thm:gsymstructureonfil}, and $\tau_{\ge0}$ has the lax $G$-symmetric monoidal structure of Proposition \ref{prop:gmonoidalslicefiltrations}. 
\end{definition}

The monoidality present on Whitehead tower functors allows one to define structured multiplication on the décalages of an Adams type filtration. We apply this to the $MU_G$-based Adams filtration in $\Sp^G$ using our two variants above of the Whitehead tower.

\begin{definition}\label{def:decalage}
We define the filtrations
\[\mathrm{AN}_G(\mathbb{S}):=\mathrm{Tot}(\mathrm{Wh}(MU_G)\implies\mathrm{Wh}(MU_G\otimes MU_G)\Rrightarrow\cdots)\]
in $\mathrm{CAlg}(\Fil(\Sp^G))$,
\[\mathrm{AN}^{RO}_G(\mathbb{S}):=\mathrm{Tot}(\mathrm{Wh}^{RO}(MU_G)\implies\mathrm{Wh}^{RO}(MU_G\otimes MU_G)\Rrightarrow\cdots)\]
in $\mathrm{CAlg}_G(\Fun(\myuline{RO}^{\mathrm{op}},\myuline{\Sp}))$, and
\[\mathrm{AN}(\mathbb{S}):=\mathrm{AN}_{\{e\}}(\mathbb{S})\]
in $\mathrm{CAlg}(\Fil(\Sp))$.
\end{definition}

\subsection{Partially normed slice filtrations} There is a version of the theory of $G$-symmetric monoidal $\infty$-categories and the corresponding $G$-commutative monoids where one has precisely the norms determined by a particular $\mathbb{N}_\infty$-operad. We refer the reader to the appendices of \cite{natalie} for an implementation of such a theory, using the theory of algebraic patterns. Taking the initial $\mathbb{N}_\infty$-operad gives no nontrivial norms and hence the "pointwise monoidal" theory of symmetric monoidal parametrized $\infty$-categories and their commutative algebras. Taking the terminal $\mathbb{N}_\infty$-operad gives all norms and hence the theory of $G$-symmetric monoidal $\infty$-categories and $G$-commutative monoids discussed above. We will need to make use of this theory for the $\mathbb{N}_\infty$-operad parametrizing all norms $N_H^K$ from nontrivial subgroups $\{e\}\neq H\subset K\subset G$.

All we will need from this theory is an analogue of Proposition \ref{prop:gmonoidalslicefiltrations} for an $\mathbb{N}_\infty$-operad $\mathcal{O}$. In the following, we fix a finite group $G$, an $\mathbb{N}_\infty$-operad $\mathcal{O}$, $\myuline{I}$ a $G$-symmetric monoidal poset, a presentably $G$-symmetric monoidal $\infty$-category $\myuline{\mathcal{C}}$, and for each $H\subset G$, an $\myuline{I}(G/H)$-indexed slice filtration on $\myuline{\mathcal{C}}(G/H)$.

\begin{proposition}\label{prop:partiallynormedslicefiltration}
    For $H\subset G$, let $\{g_{\alpha,i}^H\otimes Y(i)\}_{i,\alpha}$ be the generators of 
    \[\Fun(\myuline{I}(G/H)^{\mathrm{op}},\myuline{\mathcal{C}}(G/H))_{\ge0}\]
    as in Proposition \ref{prop:slicefiltration}. Suppose that for any  subgroups $H,H'\subset G$ and any map of $G$-sets $f:G/H\to G/H'$, one has inclusions
\[f^*(\{g_{\alpha,i}^{H'}\otimes Y(i)\}_{i,\alpha})\subset \underline{\mathcal{C}}(G/H)_{f^*(i)}\]
   for all $i\in\myuline{I}(G/H')$ and that for any $H\subset H'\subset G$ such that $H\to H'$ belongs to the transfer system corresponding to $\mathcal{O}$, one has inclusions
   \[
   N_{H}^{H'}(\{g_{\alpha,i}^H\otimes Y(i)\}_{i,\alpha})\subset \myuline{\mathcal{C}}(G/H')_{N_H^{H'}(i)}
   \]
   for all $i\in\myuline{I}(G/H)$. Then the assignment 
   \[\Fun(\myuline{I}^{\mathrm{op}},\myuline{\mathcal{C}})_{\ge0}(G/H):=\Fun(\myuline{I}(G/H)^{\mathrm{op}},\myuline{\mathcal{C}}(G/H))_{\ge0}\]
   defines an $\mathcal{O}$-symmetric monoidal subcategory of $\Fun(\myuline{I}^{\mathrm{op}},\myuline{\mathcal{C}})$, and 
   the connective cover functors of Proposition \ref{prop:slicefiltration} on $\Fun(\myuline{I}(G/H)^{\mathrm{op}},\myuline{\mathcal{C}}(G/H)$ assemble into a lax $\mathcal{O}$-symmetric monoidal right adjoint
   \[\tau_{\ge0}:\Fun(\myuline{I}^{\mathrm{op}},\myuline{\mathcal{C}})\to \Fun(\myuline{I}^{\mathrm{op}},\myuline{\mathcal{C}})_{\ge0}.\]
\end{proposition}
\begin{proof}
    By applying \cite[Corollary D.5]{natalie}, one may copy the proof of Proposition \ref{prop:gmonoidalslicefiltrations}.
\end{proof}

\begin{definition}\label{def:partiallynormedslicefiltration}
For $\myuline{I}$, $\underline{\mathcal{C}}$, and $\mathcal{O}$ as above, we call a collection of $\myuline{I}(G/H)$-indexed slice filtrations on $\myuline{\mathcal{C}}(G/H)$ satisfying the conditions of Proposition \ref{prop:partiallynormedslicefiltration} an $\mathcal{O}$-\emph{normed $\myuline{I}$-indexed slice filtration} on $\myuline{\mathcal{C}}$.    
\end{definition}

\section{Linear $t$-structures on filtered spectra}\label{sec:lineartstructures}
We recall the \emph{linear} $t$-structures on filtered spectra introduced in the context of synthetic spectra by the author and Davies in \cite{j1}, and  in a general filtered context by Lee--Levy \cite{leelevy}. For our purposes, we will work with $t$-structures on filtered spectra determined by lines of the form $y=\alpha x$ for $\alpha >-1$ a real number.

\begin{definition}\label{tstructure definition}
    Let the \emph{linear $t$-structure} on $\Fil(\Sp)$ with respect to the line $y=\alpha x$ be the $\Z$-indexed slice filtration on $\Sp$ with $n$-th $t$-structure generated by $\mathbb{S}^{\lceil\frac{n}{\alpha+1}\rceil}$. 
\end{definition}
We will use $\tau_{\ge n}^{y=\alpha x}$ and $\tau_{\le n}^{y=\alpha x}$ to denote the $n$-th cover and truncation functors respectively for the $t$-structure associated with $y=\alpha x$. Proposition \ref{prop:slicefiltration} yields the following.

\begin{proposition}[\cite{j1}]
    \label{tstructuresexist}
    The linear $t$-structure with respect to $y=\alpha x$ on $\Fil(\Sp)$ is an accessible $t$-structure, which is both right and left complete and compatible with the Day convolution symmetric monoidal structure on $\Fil(\Sp)$. 
    
    For any $Y\in\Fil(\Sp)$, a map $f:X\to Y$ exhibits $X$ as $\tau_{\ge0}^{y=\alpha x}Y$ if and only if $\pi_{x,y}X=0$ whenever $y>\alpha x$ and $f$ induces an isomorphism $\pi_{x,y}X\to \pi_{x,y}Y$ for all $y\le\alpha x$.
\end{proposition}

For filtered spectra satisfying the following property, connectivity with respect to a linear $t$-structure is equivalent to connectivity on associated graded.

\begin{definition}\label{def:stronglycomplete}
        We say a filtered spectrum $Y$ is \emph{strongly complete} if $Y$ is complete and, for all integers $x,y$, we have
    \[{\varprojlim}^1(\pi_{x,y}Y)/\tau^n=0\]
    so that, by the Milnor sequence,
    \[
    \pi_{x,y}Y\cong \varprojlim(\pi_{x,y}Y)/\tau^n
    \]
\end{definition}

\begin{remark}
    The condition that $Y$ is strongly complete precludes the existence of infinitely $\tau$-divisible elements in $\pi_{*,*}Y$.
\end{remark}

\begin{proposition}\label{prop:connectivitymodtau}
 Let $X\in\Fil(\Sp)$ be strongly complete. Then $X$ is connective in the $y=\alpha x$ linear $t$-structure if and only if $X/\tau$ is connective in the $y=\alpha x$ linear $t$-structure.
\end{proposition}
\begin{proof}
    By completeness we may write $X=\varprojlim X/\tau^n$, and via the cofiber sequences
    \[\Sigma^{0,-n+1}X/\tau\to X/\tau^n\to X/\tau^{n-1}\]
    one concludes that if $X/\tau$ is connective, then $X/\tau^n$ is connective for all $n$. By use of the Milnor sequence and strongly completeness, we see that $X$ is connective. Conversely, if $X$ is connective, the cofiber sequence
    \[X\to X/\tau\to \Sigma^{1,-2}X\]
    implies that $X/\tau$ is also connective.
\end{proof}

In \cite{j1}, we computed the heart of these linear $t$-structures, and we determined the effect of taking connective covers on mod $\tau$ homotopy groups in the case of vertical lines. Here we will spell out the nonvertical case, for which we first introduce some terminology.

\begin{definition}\label{def:linecrossing}
    Let $X\in\Fil(\Sp)$. For the line $y=\alpha x$, we say a differential $d_r(a)=b\neq 0$ in the $\tau$-Bockstein spectral sequence for $X$ is a \emph{line-crossing differential} if $a\in\pi_{x,y}(X/\tau)$, $y\le \alpha x$, and $y+r>\alpha (x-1)$. That is, the source of the differential lies on or below the line $y=\alpha x$, and the target lies strictly above it (see the $d_3$ leaving bidegree $(4,0)$ in Figure \ref{koanss} for an example).

    Let $f:X\to Y$ be a map of filtered spectra and let $b\in\pi_{x,y}Y$ such that $y>\alpha x$ and $[b]\neq0$. Let $l$ be the smallest positive integer such that $\tau^lb\in \pi_{x,\lfloor \alpha x\rfloor}Y$ and suppose that $\tau^lb\neq0$. In this case, we say $b$ \emph{drops} to the line $y=\alpha x$ and a lift $a\in \pi_{x,\lfloor \alpha x\rfloor}Y$ of $\tau^lb$ along $f$ is called a \emph{dropped lift} of $b$ to $X$.
\end{definition}

\begin{theorem}\label{thm:modtauofconnectivecover}
    Let $f:X\to Y$ be a map of filtered spectra. The map $f$ exhibits $X$ as $\tau_{\ge0}^{y=\alpha x}Y$ if and only if the following conditions hold.
    \begin{enumerate}
        \item $X$ is strongly complete.
        \item The filtered spectrum $X/\tau$ is connective in the linear $t$-structure for $y=\alpha x$, so that $\pi_{x,y}(X/\tau)=0$ for $y>\alpha x$.
        \item The map $\tau^{-1}f:\tau^{-1}X\to\tau^{-1}Y$ is an equivalence.
         \item For all $r\ge 2$, the map
    \begin{equation}\label{eq:fonEr}
        E_r-\tau\mathrm{BSS}(X)\to E_r-\tau\mathrm{BSS}(Y)
    \end{equation}
    has the following properties:
    \begin{enumerate}
        \item It is an iso when $y+r\le \alpha (x-1)$. If $y+r>\alpha (x-1)$ and $y\le \alpha x$, the image of this map consists precisely of those classes $b\in\pi_{x,y}(Y/\tau)$ that do not support a line-crossing differential in the $\tau$-BSS of $Y$.
        \item It is an injection in bidegree $(x,y)$ unless $y= \lfloor \alpha x\rfloor$.
    \end{enumerate}
    \end{enumerate}
    
    Suppose $Y$ is strongly complete. If $f$ exhibits $X$ as $\tau_{\ge0}^{y=\alpha x}Y$, the map $f$ has the following additional properties.
    \begin{enumerate}
        \setcounter{enumi}{4}
        \item If $b\in\pi_{x,y}Y$ drops to the line $y=\alpha x$ in the sense of Definition \ref{def:linecrossing}, then $b$ admits a unique dropped lift $a\in\pi_{x,y-l}X$, and $\tau^ka=0\iff \tau^{k+l}b=0$, for all $k\ge0$.
        \item The kernel of the map (\ref{eq:fonEr}) in bidegree $(x,\lfloor \alpha x\rfloor)$ consists precisely of the mod $\tau$-projections of
        dropped lifts of classes $b\in\pi_{x,y}Y$ to $X$ with $\tau^{l+r-2}b\neq0$.
    \item  If $b\in\pi_{x,y}Y$ drops to the line $y=\alpha x$ in the sense of Definition \ref{def:linecrossing}, and $\tau^{l+r-1}b=0$ for $r>1$ minimal with respect to this property, then there is a nonzero line-crossing differential $d_{l+r}(z)=[b]$ in the $\tau$-BSS of $Y$. The class $z$ admits a unique lift $\tilde{z}$ along (\ref{eq:fonEr}), and $d_{r}(\tilde{z})=[a]$, where $a$ is the dropped lift to $X$ of $b$. Conversely, every differential in the $\tau$-BSS of $X$ whose target is the mod $\tau$-projection of a dropped lift arises in this way. 
    \item Every differential in the $\tau$-BSS of $X$ is a lift of a differential in the $\tau$-BSS of $X$ or is of the form appearing in condition (7).
    
\end{enumerate}
\end{theorem}
\begin{proof}
    We begin with the first half of the theorem statement and suppose that conditions (1)-(4) hold. To see that $f$ must be injective on bigraded homotopy groups, suppose that $0\neq a\in\pi_{x,y}X$ and $f(a)=0$. By Proposition \ref{prop:connectivitymodtau}, conditions (1) and (2) guarantee that $X$ is connective in the linear $t$-structure for $y=\alpha x$, so we may assume wlog that $y\le \alpha x$. Since $\tau^{-1}f$ is injective by condition (3), there exists $r>0$ such that $\tau^ra=0$ and we may suppose wlog that $r$ is minimal with respect to this property. By exactness, there exists $z\in\pi_{x+1,y-r-1}X/\tau$ with $\delta(z)=\tau^{r-1}a$. Since $y-r-1<\lfloor\alpha(x+1)\rfloor$, by condition (4b) we have that $f(z)\neq0$, and since $\delta(f(z))=0$, the class $f(z)$ is a nonzero permanent cycle. Now, using connectivity of $X$ to write $a=\tau^k\alpha$ with $k\ge0$ and $[\alpha]\neq0$, we have that $d_{r+k+1}(z)=[\alpha]$, so $f(z)$ is not in the image of $f$ on the $E_{r+k+2}$-page. By (4a), $f(z)$ must therefore support a line crossing differential, contradicting that it is a permanent cycle.

    To see that $f$ is surjective in bidegrees $(x,y)$ with $y\le\alpha x$, let $b\in\pi_{x,y}Y$ with $y\le\alpha x$ and suppose that $\tau^{-1}b\neq0$. Since $\tau^{-1}f$ is surjective by condition (3), there exists $a\in\pi_{x,y-k}X$ with $f(a)=\tau^kb$ and we may take $k>0$ wlog. Using connectivity of $X$ to write $a=\tau^r\alpha$ with $r\ge0$ and $[\alpha]\neq0$, we have $f(\tau^r\alpha)=\tau^kb$. Suppose first that $r\ge k$, then it follows that $\tau^k(b-\tau^{r-k}f(\alpha))=0$, so there exists $z\in \pi_{x+1,y-k-1}Y/\tau$ such that $\delta(z)=\tau^{k-1}(b-\tau^{r-k}f(\alpha))$. The class $z$ is on or below the line $y+2\le\alpha(x-1)$, so by (4a), it admits a lift $\tilde{z}$ to $X/\tau$. We have $f(\delta(\tilde{z}))=\tau^{k-1}(b-\tau^{r-k}f(\alpha))$, so it follows that $\tau^{k-1}b$ is in the image of $f$. Repeating the argument with $\tau^{k-1}b$ in place of $\tau^kb$, we may proceed until $k=0$, and therefore $b$ is in the image of $f$. A similar argument works when $r<k$.

    Suppose then that $b\in\pi_{x,y}Y$ with $y\le\alpha x$, now with $\tau^{k+1}b=0$, and let $k$ be minimal with respect to this property. As before, there exists $z$ with $\delta(z)=\tau^{k}b$, and $z$ admits a lift $\tilde{z}$ to $X/\tau$ by (4a). But then $\tau^kb=f(\delta(\tilde{z}))$ is in the image of $f$, and we may produce a lift of $b$ exactly as in the previous paragraph.

    For the converse, assume that $f$ exhibits $X$ as the connective cover. Conditions (1) and (2) follow for degree reasons from the fact that $X$ is connective in the linear $t$-structure for $y=\alpha x$. For (3), the cofiber $\tau_{\le -1}^{y=\alpha x}Y$ of $f$ has $\pi_{x,y}\tau_{\le -1}^{y=\alpha x}Y=0$ for $y\le \alpha x$. Since $\tau$ has bidegree $(0,-1)$, it follows that every element of $\pi_{*,*}\tau_{\le -1}^{y=\alpha x}Y$ is $\tau$-power torsion, so $\tau^{-1}\tau_{\le -1}^{y=\alpha x}Y=0$ and $\tau^{-1}f$ is an equivalence. For (4), we begin with the case $r=2$. Consider the diagram of long exact sequences
        \[
    \begin{tikzcd}
        \pi_{x,y+1}X\arrow[r]\arrow[d]&\pi_{x,y}X\arrow[r]\arrow[d]&\pi_{x,y}(X/\tau)\arrow[r,"\delta"]\arrow[d]&\pi_{x-1,y+2}X\arrow[d]\arrow[r,"\tau"]&\pi_{x-1,y+1}X\arrow[d]\\
        \pi_{x,y+1}Y\arrow[r]&\pi_{x,y}Y\arrow[r]&\pi_{x,y}(Y/\tau)\arrow[r,"\delta"]&\pi_{x-1,y+2}Y\arrow[r,"\tau"]&\pi_{x-1,y+1}Y
    \end{tikzcd}
    \]
    induced by the multiplication by $\tau$ cofiber sequence. When $y+2\le \alpha (x-1)$, all vertical maps except the middle one are isomorphisms by assumption, hence so is the middle by the five lemma. If $a\in\pi_{x,y}(X/\tau)$ for $y+2>\alpha (x-1)$, then $a$ is a permanent cycle in the $\tau$-BSS for $X$ for degree reasons, as $X/\tau$ is connective. By naturality of the $\tau$-BSS, $f(a)$ is a permanent cycle in the $\tau$-Bockstein SS for $Y$. Conversely, if $b\in \pi_{x,y}(X/\tau)$ for $y+2>\alpha (x-1)$ and $y\le \alpha x$ does not support a line crossing differential, then it is a permanent cycle, so it admits a lift $\tilde{b}\in\pi_{x,y}Y$ along the $\tau$-reduction map. Since $\pi_{x,y}X\to \pi_{x,y}Y$ is an iso in this bidegree, naturality of the mod $\tau$ projection map now implies $b$ is in the image of $f$, as desired. This proves (4a) in the case $r=2$.    

    For (4b), suppose 
    \[a\in\ker(\pi_{x,y}(X/\tau)\to \pi_{x,y}(Y/\tau))\]
    for $y+2>\alpha (x-1)$ and $y<\lfloor \alpha x\rfloor$. For degree reasons, $a$ must be a permanent cycle, so it admits a lift $\tilde{a}\in\pi_{x,y}X$ along the $\tau$-reduction map such that $f(\tilde{a})$ is divisible by $\tau$ by assumption. But since $y<\lfloor \alpha x\rfloor$, $y+1\le \alpha x$ and hence $\pi_{x,y+1}X\to \pi_{x,y+1}Y$ is an iso, so that $\tilde{a}$ must also be divisible by $\tau$. Therefore $a=0$ as desired. The proof of (4) for $r>2$ now proceeds by induction using using what we have just proved as a base case and analogous arguments.

    For the second part of the theorem, we assume now that $Y$ is strongly complete and that $f$ exhibits $X$ as the connective cover.  Condition (5) follows from the fact that $\pi_{x,y}X\to \pi_{x,y}Y$ is an isomorphism for all $y\le\alpha x$. For (6), suppose that $a\in\pi_{x,\lfloor \alpha x\rfloor}(X/\tau)$ is in the kernel of (\ref{eq:fonEr}). The unique lift $\tilde{a}\in\pi_{x,\lfloor \alpha x\rfloor}X$ of $a$ has the property that $f(\tilde{a})$ is divisible by $\tau$, since $[f(\tilde{a})]=0$ by assumption. It follows that $\tilde{a}$ is a dropped lift of a class $b$ in $Y$ since we may use strong completeness of $Y$ to produce the class $b$ with $[b]\neq0$ by dividing $f(\tilde{a})$ by sufficiently many powers of $\tau$. The fact that $a$ is nonzero on $E_r$ implies that $\tau^{r-1}\tilde{a}\neq0$, so by condition (6), we see that $\tau^{l+r-2}b\neq0$. Conversely, every dropped lift $a$ of a class $b$ in $Y$ is the kernel of $f$ mod $\tau$ by definition as $l>0$. For degree reasons $[a]$ must be a permanent cycle and it is hit by a $d_{r'}$ for $r'<r$ if and only if $\tau^{l+r-2}b=0$. Therefore $a$ determines a nonzero class on $E_r$ if $\tau^{l+r-2}b\neq0$.

    For (7), the first statement follows from the usual relationship between $\tau$-power torsion and differentials in the $\tau$-BSS. The class $z$ is in the region of $E_r$ in which (\ref{eq:fonEr}) is an isomorphism, hence it admits the lift $\tilde{z}$. One has $\delta(z)=\tau^{l+r-2}b$ so that $f(\delta(\tilde{z}))=\tau^{l+r-2}b$ , so that $\delta(\tilde{z})=\tau^{r-2}a$ where $a$ is the dropped lift of $b$ to $X$, by injectivity of $f$ in this region. This further implies the differential $d_r(\tilde{z})=[a]$. Conversely, if there is a differential $d_r(\tilde{z})=[a]$ in the $\tau$-BSS of $X$, for $a$ a dropped lift of $b$ in $Y$, then by connectivity of $X$, one has $\delta(\tilde{z})=\tau^{r-2}a$, so that $\delta(f(\tilde{z}))=\tau^{r+l-2}b$. The integer $r$ is minimal with respect to this latter property by the second half of condition (5).

    For (8), suppose there is a differential $d_r(z)=[a]$ in the $\tau$-BSS of $X$ that is not a lift of a $d_r$ in the $\tau$-BSS of $Y$. This means that $f([a])=0$, so by (4b) and (6), the class $a$ must be a dropped lift of a class $b$ in Y, as desired.\end{proof}

\begin{remark}
    We highlight separately the $E_2$-page case of Theorem \ref{thm:modtauofconnectivecover} conditions (4) and (6) since this describes the behavior of the map $f$ on the homotopy groups of the associated graded of these filtrations. We assume that $f:X\to Y$ exhibits $X$ as the connective cover of $Y$ in the $y=\alpha x$ linear $t$-structure and that $Y$ is strongly complete. We then have the following.
    \begin{enumerate}
        \item The map
        \[
        \pi_{x,y}(X/\tau)\to \pi_{x,y}(Y/\tau)
        \]
        is an iso when $y+2\le \alpha (x-1)$, i.e. when the bidegree $(x-1,y+2)$ is on or below the line $y=\alpha x$. If $y+2> \alpha (x-1)$ and $y\le \alpha x$, the image of this map consists precisely of those classes $b\in\pi_{x,y}(Y/\tau)$ that are permanent cycles in the $\tau$-Bockstein SS for $Y$.
        \item The map
        \[
        \pi_{x,y}(X/\tau)\to \pi_{x,y}(Y/\tau)
        \]
        is injective unless $y=\lfloor \alpha x\rfloor$. The kernel of the map in bidegree $(x,\lfloor\alpha x\rfloor)$ consists precisely of the mod $\tau$-projections of dropped lifts of classes $b\in\pi_{x,y}Y$ to $X$, and every class $b\in\pi_{x,y}Y$ that drops to the line $y=\alpha x$ admits a unique such dropped lift.
    \end{enumerate}
\end{remark}

We give an example that illustrates the above theorem, where $Y=\mathrm{AN}(ko)$ is the Adams--Novikov filtration of $ko$, which one may implement as a filtered spectrum via the \emph{signature} functor $\sigma:\mathrm{Syn}_{MU}\to\Fil(\Sp)$ discussed in \cite[Section 1.4]{CDvN}.  Here we take $\alpha=1/2$, and in Figure \ref{koanss}, we show the ANSS for $ko$, with the line $y=(1/2)x$ drawn in green. In Figure \ref{kocover}, we show the $\tau$-Bockstein SS for $\tau_{\ge0}^{y=(1/2)x}\mathrm{AN}(ko)$.

The lines $y=(1/2)x$ and $y+2=(1/2)(x-1)$ are drawn in these figures in green. The black squares (copies of $\Z$) and dots (copies of $\F_2$) in Figure \ref{kocover} map isomorphically to the black squares and dots in Figure \ref{koanss} via the map $f:X\to Y$, the blue square in Figure \ref{kocover} in bidegree $(4,0)$ maps to the black square in Figure \ref{koanss} in bidegree $(4,0)$ via multiplication by $2$. The orange dots in Figure \ref{kocover} in $(1,0),(2,1),(17,8),(18,9)$ are the mod $\tau$ projections of the dropped lifts of the black dots in Figure \ref{koanss} in $(1,1),(2,2),(17,9),(18,10)$, respectively. The two orange $d_2$-differentials leaving bidegrees $(18,6)$ and $(19,7)$ respectively are those of the form described in Proposition \ref{thm:modtauofconnectivecover} (7). Finally, structure lines indicate hidden $\eta$-multiplications of the appropriate filtration jumps, where $\eta\in\pi_{1,-1}\mathbbm{1}$.

\begin{remark}
    The class $1\in \pi_{0,0}\mathrm{AN}(ko)$ lifts to $\pi_{0,0}(\tau_{\ge0}^{y=(1/2)x}\mathrm{AN}(ko))$, but $\eta\in \pi_{1,1}\mathrm{AN}(ko)$ does not. This implies that the map of filtered spectra \[\tau_{\ge0}^{y=(1/2)x}\mathrm{AN}(ko)\to \mathrm{AN}(ko)\]
    does not lift to one of modules over $\mathrm{AN}(\mathbb{S})$, i.e. to one of $MU$-synthetic spectra. This reflects the fact that $\mathrm{AN}(\mathbb{S})$ is not connective in the linear $t$-structure on $\Fil(\Sp)$ associated to $y=(1/2)x$.
\end{remark}

\NewSseqGroup\etatowers {} {
    \foreach \n in {0,...,20} {
    \class(\n+1,\n+1)
    }
    \foreach \m in {0,...,19} {
    \structline[red](\m,\m)(\m+1,\m+1)
    }
}

\NewSseqGroup\etatowersblue {} {
    \foreach \n in {0,...,20} {
    \class[blue](\n+1,\n+1)
    }
    \foreach \m in {0,...,19} {
    \structline[black](\m,\m)(\m+1,\m+1)
    }
}
\NewSseqGroup\etatowersred {} {
    \foreach \n in {0,...,20} {
    \class[red](\n+1,\n+1)
    }
    \foreach \m in {0,...,19} {
    \structline[black](\m,\m)(\m+1,\m+1)
    }
}

\begin{sseqdata}[name = ansskoattwo, Adams grading, classes = {fill, show name=below},
grid = go, xrange ={0}{20},yrange={0}{10},xscale=0.35,yscale=0.7,x tick step =2, y tick step =2,run off differentials = {->},struct lines = red ]

\class[rectangle](0,0)\class[rectangle,"u^2" {right,xshift=-0.25cm,yshift=-0.18cm}](4,0)\class[rectangle](8,0)\class[rectangle](12,0)\class[rectangle](16,0)\class[rectangle](20,0)
\class(-4,-2)
\class(30,15)
\structline[color=green](-4,-2)(30,15)
\class["\eta" {right,xshift=-0.25cm,yshift=-0.18cm}](1,1)
\structline(0,0)(1,1)
\etatowers(1,1)
\etatowers(4,0)
\etatowers(8,0)
\etatowers(12,0)
\etatowers(16,0)
\etatowers(20,0)
\foreach \i in {0,...,10} {
    \d3(4+\i,\i)
    }

\foreach \i in {0,...,10} {
    \d3(12+\i,\i)
    }

\foreach \i in {0,...,10} {
    \d3(20+\i,\i)
    }
\end{sseqdata}

\begin{figure}[!htbp]
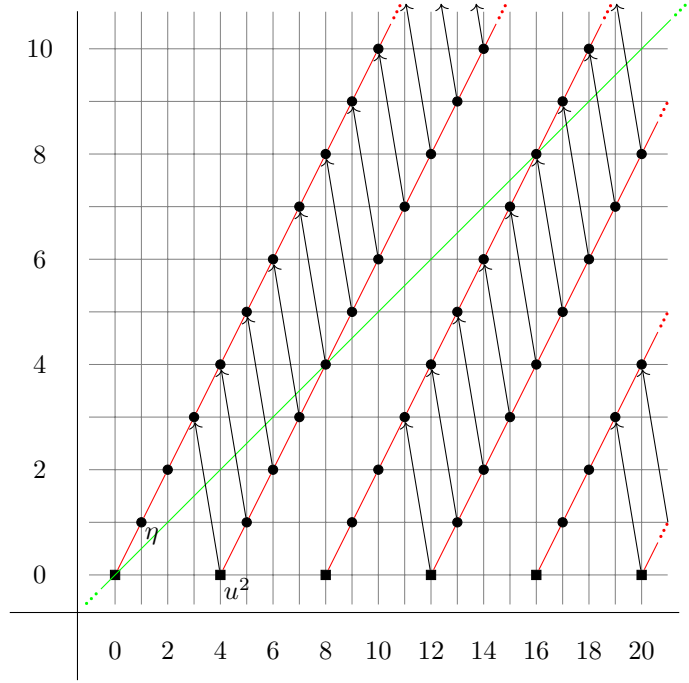

\centering
\printpage[name = ansskoattwo]
\caption{The ANSS for $ko$.}
\label{koanss}
\end{figure}

\begin{sseqdata}[name = ansscover, Adams grading, classes = {fill, show name=below},
grid = go, xrange ={0}{21},yrange={0}{10},xscale=0.35,yscale=0.7,x tick step =2, y tick step =2,run off differentials = {->},struct lines = red ]
\class[rectangle](0,0)
\class(-4,-2)
\class(30,15)
\structline[color=green](-4,-2)(30,15)
\class(3,-1)
\class(23,9)
\structline[color=green](3,-1)(23,9
)
\class[color=orange](1,0)
\class[color=orange](2,1)
\class[rectangle,color=blue](4,0)
\structline(0,0)(1,0
)
\structline(1,0)(2,1)
)

\class[rectangle](8,0)\class[rectangle](12,0)\class[rectangle](16,0)\class[rectangle](20,0)
\etatowers(16,0)
\etatowers(20,0)
\d3(20,0)
\d3(21,1)
\d3(22,2)

\class(13,1)\class(14,2)\class(15,3)\class(16,4)\class(17,5)\class(18,6)\class(19,7)
\structline(12,0)(13,1)
\structline(13,1)(14,2)
\structline(14,2)(15,3)
\structline(15,3)(16,4)
\structline(16,4)(17,5)
\structline(17,5)(18,6)
\structline(18,6)(19,7)

\class(9,1)\class(10,2)\class(11,3)\class(12,4)\class(13,5)\class(14,6)\class(15,7)\class(16,8)
\structline(8,0)(9,1)
\structline(9,1)(10,2)
\structline(10,2)(11,3)
\structline(11,3)(12,4)
\structline(12,4)(13,5)
\structline(13,5)(14,6)
\structline(14,6)(15,7)
\structline(15,7)(16,8)

\d3(12,0)\d3(13,1)\d3(14,2)\d3(15,3)\d3(16,4)\d3(17,5)

\class[color=orange](17,8)
\class[color=orange](18,9)
\structline(16,8)(17,8
)
\structline(17,8)(18,9)
)

\d[orange]2(18,6)\d[orange]2(19,7)
\end{sseqdata}

\begin{figure}[!htbp]
\centering
\printpage[name = ansscover]
\caption{The $\tau$-BSS for $\tau_{\ge0}^{y=(1/2)x}\mathrm{AN}(ko)$.}
\label{kocover}
\end{figure}

We end the section with a simple criterion for checking that a map of filtered spectra induces an equivalence on connective covers, which we will apply in our comparison of the slice and homotopy fixed point filtrations.

\begin{proposition}\label{prop:sufficientforconnectivecover}
Let $f:X\to Y$ be a map of filtered spectra. If $f$ induces equivalences
\begin{enumerate}
    \item $
    \tau_{\ge\lceil \frac{-n}{\alpha +1}\rceil}(X^{-n})\to \tau_{\ge\lceil \frac{-n}{\alpha +1}\rceil}(Y^{-n})$ for all $n$ sufficiently large
    \item $\tau_{\ge0}^{y=\alpha x}(X/\tau)\to \tau_{\ge0}^{y=\alpha x}(Y/\tau)$,
\end{enumerate}
    then $f$ induces an equivalence
    \[
    \tau_{\ge0}^{y=\alpha x}X\to \tau_{\ge0}^{y=\alpha x}Y.
    \]
\end{proposition}
\begin{proof}
    By Proposition \ref{prop:slicefiltration} (2), it suffices to show that $f$ induces an equivalence
    \[
    \tau_{\ge\lceil \frac{n}{\alpha +1}\rceil}(X^n)\to \tau_{\ge\lceil \frac{n}{\alpha +1}\rceil}(Y^n)
    \]
    for all $n$. By the first assumption this holds for all sufficiently small $n$, so we may assume by induction that $f$ induces an equivalence
    \[
    \tau_{\ge\lceil \frac{n-1}{\alpha +1}\rceil}(X^{n-1})\to \tau_{\ge\lceil \frac{n-1}{\alpha +1}\rceil}(Y^{n-1}).
    \]
    
    Fix $k\ge \lceil \frac{n}{\alpha +1}\rceil$ and consider the morphism of long exact sequences
    \[
\adjustbox{center,scale=.94}{
\begin{tikzcd}
    \pi_{k+1}X^{n-1}\arrow[r]\arrow[d]&\pi_{k+1,n-2-k}(X/\tau)\arrow[r]\arrow[d]&\pi_kX^n\arrow[r]\arrow[d]&\pi_kX^{n-1}\arrow[r]\arrow[d]&\pi_{k,n-1-k}(X/\tau)\arrow[d]\\
    \pi_{k+1}Y^{n-1}\arrow[r]&\pi_{k+1,n-2-k}(Y/\tau)\arrow[r]&\pi_kY^n\arrow[r]&\pi_kY^{n-1}\arrow[r]&\pi_{k,n-1-k}(Y/\tau).
\end{tikzcd}
    }
    \]
    The first and fourth vertical arrows are isomorphisms by the inductive hypothesis, and the second and fifth are isomorphisms by assumption (2), so the result follows by the five lemma.
\end{proof}

\begin{remark}
    In practice, we will verify condition (1) of Proposition \ref{prop:sufficientforconnectivecover} in the case that $X$ and $Y$ are \emph{eventually constant} in the sense that there exists $N\in\Z$ such that for all $n\ge N$, the map $X^{-n}\to X^{-n-1}$ is an equivalence. In this case, it suffices to show that $f$ induces an equivalence $\tau^{-1}X\to \tau^{-1}Y$ to establish condition (1).

    If $X$ is a connective spectrum and $E$ a connective Adams-type spectrum, the $E$-based Adams spectral sequence of $X$, regarded as a filtered spectrum via the signature functor $\sigma:\mathrm{Syn}_{E}\to\Fil(\Sp)$ (see \cite[Section 1.4]{CDvN}), is eventually constant. In fact, here $N$ may be taken to be zero. Other examples include (variants of) Whitehead towers for bounded below objects.
\end{remark}

\section{The slice filtration as a connective cover}\label{sec:slicefiltrationisconnectivecover}
In this section, we complete the proofs of our main theorems, by showing that the slice filtration is obtained as a connective cover of the homotopy fixed point filtration, and we use this to construct canonical maps to the slice filtration.

\subsection{The slice filtration and $t$-structures on $\Fil(\Sp^G)$} We will define two $\myuline{\Z}$-indexed slice filtrations on $\myuline{\Sp}$, in the sense of Definition \ref{def:normedslicefiltration}. By evaluating at $G/G$, these give $\Z$-indexed slice filtrations on $\Sp^G$, in the sense of Definition \ref{def:slicefiltration}. In what follows, we let $\mathcal{O}$ denote the $\mathbb{N}_\infty$-operad corresponding to the transfer system with all norms $H\to H'$ for $\{e\}\neq H$.

\begin{definition}\label{def:twofilteredtstructures}
    Let the \emph{slice $t$-structure} on $\Fil(\myuline{\Sp})$ be the $\myuline{\Z}$-indexed slice filtration on $\myuline{\Sp}$ defined by, for each $i\in\Z$ and $H\subset G$, the $t$-structure on $\Sp^H$ generated by 
    \[\{\mathrm{Ind}_L^H\mathbb{S}^{k\rho_L}:k|L|\ge i\}_{L\subset H}\]
    Let the $\mathcal{O}$-\emph{slice $t$-structure} on $\Fil(\myuline{\Sp})$ be the $\myuline{\Z}$-indexed slice filtration on $\myuline{\Sp}$ defined by, for each $i\in\Z$ and $H\subset G$, the $t$-structure on $\Sp^H$ generated by 
    \[\{\mathrm{Ind}_e^H\mathbb{S}^n\}_{n\in\Z}\cup\{\tilde{E}G\otimes \mathrm{Ind}_L^H\mathbb{S}^{k\rho_L}:k|L|\ge i\}_{L\subset H}\]
    We use the notation $\tau_{\ge0}^{\mathrm{slice}}$ and $\tau_{\ge0}^{\mathcal{O}}$ respectively to denote the connective cover functors of these two $t$-structures.
\end{definition}

\begin{remark}\label{rmk:geometricconnectivity}
    In terms of the \emph{geometric connectivity} perspective of Hill--Yarnall \cite{hillyarnall} (see also \cite[Theorem 11.1.27]{HHRbook}), $X\in\Fil(\Sp^G)$ is connective in the slice $t$-structure if and only if $\Phi^HX\in\Fil(\Sp)$ is connective in the $y=(|H|-1)x$ linear $t$-structure for all $H\subset G$. A filtered $G$-spectrum  $X\in\Fil(\Sp^G)$ is connective in the $\mathcal{O}$-slice $t$-structure if and only if $\Phi^HX\in\Fil(\Sp)$ is connective in the $y=(|H|-1)x$ linear $t$-structure for all $\{e\}\neq H\subset G$. Indeed, this follows immediately from \emph{loc. cit.} since these filtered $t$-structures are defined pointwise in terms of the ordinary slice filtration. Alternatively, this may be argued directly by induction on $|G|$ using isotropy separation, as in \emph{loc. cit.} (see also \cite[Chapter 11]{HHRbook}).
\end{remark}

These slice filtrations also have the expected (partially) normed structures.

\begin{proposition}
    The slice $t$-structure defines a normed $\myuline{\Z}$-indexed slice filtration on $\Fil(\myuline{\Sp})$ in the sense of Definition \ref{def:normedslicefiltration}. The $\mathcal{O}$-slice $t$-structure defines an $\mathcal{O}$-normed $\myuline{\Z}$-indexed slice filtration on $\Fil(\myuline{\Sp})$ in the sense of Definition \ref{def:partiallynormedslicefiltration}. 
\end{proposition}
\begin{proof}
    It suffices to check the conditions of Propositions \ref{prop:gmonoidalslicefiltrations} and \ref{prop:partiallynormedslicefiltration} respectively, on the generators of these $t$-structures, which follow immediately from Lemma \ref{lemma:normsofbigradedspheres} and the distributive law  of \cite[Proposition A.37]{HHR}, which expresses norms of indexed coproducts in terms of indexed coproducts of norms (cf. \cite[Proposition 11.1.10]{HHRbook}).
\end{proof}

As in Example \ref{example:slicefiltration}, the slice $t$-structure recovers the usual equivariant slice filtration of a $G$-spectrum in the following manner.

\begin{definition}\label{def:slicetower}
The slice filtration $\mathrm{Slice}(X)$ of $X\in\Sp^G$ is the image of $X$ under the composition
\[\Sp^G\xrightarrow{\mathrm{const}}\Fil(\Sp^G)\xrightarrow{\tau_{\ge0}^{\mathrm{slice}}}\Fil(\Sp^G)\]
Since $\mathrm{const}$ is lax $G$-symmetric monoidal, and the slice $t$-structure defines a normed slice filtration, if $X$ has a $G$-$\mathbb{E}_\infty$-structure, so does $\mathrm{Slice}(X)$. 
\end{definition}

Our main result of this section is a comparison of the slice and homotopy fixed point spectral sequences in a filtered context. We first define the homotopy fixed point filtration and discuss its multiplicative structure. In the following, we let $F({EG}_+,-):\Fil(\Sp^G)\to\Fil(\Sp^G)$ be the functor given by applying $F({EG}_+,-)$ pointwise.

\begin{definition}\label{def:hfpsss}
    The homotopy fixed point filtration $\mathrm{HFP}(X)$ of $X\in\Sp^G$ is the image of $X$ under the composite
    \[\Sp^G\xrightarrow{\mathrm{Wh}}\Fil(\Sp^G)\xrightarrow{F({EG}_+,-)}\Fil(\Sp^G)\]
    where Wh denotes the Whitehead tower functor of Definition \ref{def:whitehead}.
\end{definition}

As discussed in Section \ref{subsec:ROwhitehead}, the Whitehead tower functor is not lax $G$-symmetric monoidal, so one does not expect a priori a lax $G$-symmetric monoidal structure on $\mathrm{HFP}(-)$. However, one does have this structure due to the "monoidal Borelification principle", whereby Borelification $F({EG}_+,-)$ turns naive commutative data into genuine commutative data. This phenomenon was observed by Blumberg--Hill \cite[Theorem 6.26]{blumberghill} and was codified into the setting of $G$-symmetric monoidal $\infty$-categories by Hilman \cite[Section 2.4]{kaif}. In particular, using the naive $\mathbb{E}_\infty$-structure on $\mathrm{HFP}(-)$ and applying \cite[Theorem 2.4.10]{kaif} gives the following.

\begin{proposition}\label{prop:hfpssisGsymmonoidal}
    For $R\in\Sp^G$ a naive $\mathbb{E}_\infty$-algebra, the $\mathbb{E}_\infty$ structure on $\mathrm{HFP}(R)$ refines canonically to a $G$-$\mathbb{E}_\infty$ structure.
\end{proposition}
\begin{proof}
    The Borelification adjunction 
    \[
    \begin{tikzcd}
    \Fil(\myuline{\Sp})\arrow[rr,shift left]&&\mathrm{Bor}(\Fil(\myuline{\Sp}))\arrow[ll,shift left,"F({EG}_+\text{,}-)"]
    \end{tikzcd} 
    \]
    of \cite[Proposition 2.4.1]{kaif} refines canonically to a $G$-symmetric monoidal adjunction by \cite[Theorem 2.4.10(2)]{kaif}, and there is an equivalence
    \[\mathrm{CAlg}_G(\mathrm{Bor}(\Fil(\myuline{\Sp})))\simeq\mathrm{CAlg}(\Fil(\Sp^{BC_2}))\]
    by \cite[Theorem 2.4.10(3)]{kaif}. The $\mathbb{E}_\infty$-structure on $R$ canonically defines the object $\mathrm{HFP}(R)\in \mathrm{CAlg}(\Fil(\Sp^{BC_2}))$. The desired $G$-$\mathbb{E}_\infty$ structure is then given by regarding $\mathrm{HFP}(R)$ as an object in $\mathrm{CAlg}_G(\mathrm{Bor}(\Fil(\myuline{\Sp})))$ via the above equivalence and applying the right adjoint of the above Borelification adjunction.
\end{proof}

In fact, there is a map of $\mathbb{E}_\infty$-algebras $\mathrm{Slice}(R)\to \mathrm{HFP}(R)$. This map was defined by Ullman \cite[Theorem 9.1]{ullman}, and we follow his ideas to construct this map in our context. For this we will need another filtration, which Ullman calls the \emph{mixed} filtration, to interpolate between $\mathrm{HFP}(-)$ and $\mathrm{Slice}(-)$.

\begin{definition}
    Let the \emph{mixed $t$-structure} on $\Fil(\Sp^G)$ be the $\Z$-indexed slice filtration on $\Sp^G$ defined by, for $i\in\Z$, the $t$-structure on $\Sp^G$ generated by 
    \[\{\mathrm{Ind}_H^G\mathbb{S}^i\}_{H\subset G}\cup \{\mathrm{Ind}_H^G\mathbb{S}^{m\rho_H}|m|H|\ge i\}_{H\subset G}.\]
    Let $\tau_{\ge0}^{\mathrm{mixed}}$ denote the corresponding connective cover functor on $\Fil(\Sp^G)$. For $X\in\Sp^G$, define the mixed filtration $\mathrm{Mixed}(X)$ of $X$ to be the image of $X$ under the composite
    \[\Sp^G\xrightarrow{\mathrm{const}}\Fil(\Sp^G)\xrightarrow{\tau_{\ge0}^{\mathrm{mixed}}}\Fil(\Sp^G)\]
\end{definition}

The mixed filtration is an auxiliary filtration defined so that we can prove the following proposition.

\begin{proposition}\label{prop:spanwithmixedfiltration}
    For any $X\in\Sp^G$, there exists a span in $\Fil(\Sp^G)$
    \[\mathrm{HFP}(X)\to\mathrm{Mixed}(X)\leftarrow\mathrm{Slice}(X)\]
    such that both maps become equivalences upon applying Borelification $F({EG}_+,-)$. If $X$ has a naive $\mathbb{E}_\infty$-structure, this span lifts to $\mathrm{CAlg}(\Fil(\Sp^G))$.
\end{proposition}
\begin{proof}
    Define
    \begin{align*}
        \mathcal{C}_i^{\mathrm{Wh}}&:=\langle\{\mathrm{Ind}_H^G\mathbb{S}^i\}_{H\subset G}\rangle\\
        \mathcal{C}_i^{\mathrm{Slice}}&:=\langle\{\mathrm{Ind}_H^G\mathbb{S}^{m\rho_H}:m|H|\ge i\}_{H\subset G}\rangle\\
        \mathcal{C}_i^{\mathrm{mixed}}&:=\langle\{\mathrm{Ind}_H^G\mathbb{S}^i,\mathrm{Ind}_H^G\mathbb{S}^{m\rho_H}:m|H|\ge i\}_{H\subset G}\rangle.
    \end{align*}
    By definition, there are inclusions $\mathcal{C}_i^{\mathrm{Slice}}\subset \mathcal{C}_i^{\mathrm{mixed}}$ and $\mathcal{C}_i^{\mathrm{Wh}}\subset \mathcal{C}_i^{\mathrm{mixed}}$ for all $i\in \Z$. These inclusions determine corresponding inclusions of connective objects in the respective $t$-structures defined on $\Fil(\Sp^G)$ by the $\mathcal{C}_i$'s. By composition of adjoints, one has the above span. When $X$ has an $\mathbb{E}_\infty$ structure, the span lifts to span in $\mathbb{E}_\infty$-algebras since each of the inclusions is symmetric monoidal, so that they lift to left adjoints on categories of commutative algebras. Each of the maps induces an equivalence on underlying spectra because, for all $i\in\Z$, one checks that $\mathrm{Res}^G_e(\mathcal{C}_i^{\mathrm{Slice}})=\mathrm{Res}^G_e(\mathcal{C}_i^{\mathrm{Wh}})$ using that $\mathrm{Res}^G_e$ is a left adjoint. It follows that the maps induce equivalences upon Borelification. 
\end{proof}

\begin{corollary}\label{cor:mapfromslicetohfpss}
    For any $X\in\Sp^G$, there exists a map in $\Fil(\Sp^G)$
    \[\mathrm{Slice}(X)\to\mathrm{HFP}(X)\]
    inducing an equivalence upon Borelification. If $X$ has a $G$-$\mathbb{E}_\infty$-structure, this map lifts to $\mathrm{CAlg}_G(\Fil(\myuline{\Sp}))$.
\end{corollary}
\begin{proof}
    One has the Borelification map 
    \[\mathrm{Slice}(X)\to F({EG}_+,\mathrm{Slice}(X)),\]
    which lifts to $\mathrm{CAlg}_G(\Fil(\Sp^G))$ when $X$ has a $G$-$\mathbb{E}_\infty$-structure, by \cite[Theorem 2.4.10(2)]{kaif}. By the previous proposition, there is an equivalence 
\[
        F({EG}_+,\mathrm{Slice}(X))\simeq \mathrm{HFP}(X),
\]
        which gives the claimed map, and this is an equivalence of naive $\mathbb{E}_\infty$-algebras
when $X$ has an $\mathbb{E}_\infty$-structure. By applying the monoidal Borelification principle as in the proof of Proposition \ref{prop:hfpssisGsymmonoidal}, this refines canonically to an equivalence of $G$-$\mathbb{E}_\infty$-algebras.
\end{proof}

The map
\[\mathrm{Slice}(X)\to\mathrm{HFP}(X)\]
factors canonically through
\[\mathrm{Slice}(F(EG_+,X))\to\mathrm{HFP}(X)\]
via the diagram
\[
\begin{tikzcd}
    \mathrm{Slice}(X)\arrow[r]\arrow[d]&\mathrm{HFP}(X)\arrow[d,"\simeq"]\\
    \mathrm{Slice}(F(EG_+,X))\arrow[r]&\mathrm{HFP}(F(EG_+,X))
\end{tikzcd}
\]
We show now that the map $\mathrm{Slice}(F(EG_+,X))\to\mathrm{HFP}(X)$ is actually a connective cover in the slice $t$-structure.

\begin{lemma}\label{lemma:sliceandhfpss}
     For all $n\in\Z$ and $X\in\Sp^G$, the natural map $\mathrm{Slice}^n(X)\to X$ induces an equivalence
    \[
    \mathrm{Slice}^n(F(EG_+,\mathrm{Slice}^n(X)))\to \mathrm{Slice}^n(F(EG_+,X))
    \]
\end{lemma}
\begin{proof}
    It suffices to show that if $Z\in\Sp^G$ is slice $\ge n$, then 
    \[[Z,F(EG_+,\mathrm{Slice}^n(X))]_{\Sp^G}\to [Z,F(EG_+,X)]_{\Sp^G}\]
    is an iso. By Corollary \ref{cor:mapfromslicetohfpss}, we may replace $\mathrm{Slice}^n(X)$ with $\tau_{\ge n}X$ here, and
    the skeletal filtration of the $G$-space $EG_+$ reduces this to showing that the map 
    \[[Z,F(\Sigma^kG_+,\tau_{\ge n}X)]_{\Sp^G}\to [Z,F(\Sigma^kG_+,X)]_{\Sp^G}\]
    is an iso for $k\ge0$, which by adjunction is the map
    \[[\Sigma^k\mathrm{Res}^G_eZ,\tau_{\ge n}X^e]_{\Sp}\to [\Sigma^k\mathrm{Res}^G_eZ,\mathrm{Res}^G_eX]_{\Sp}.\]
    This map is an iso because the fact that $Z$ is slice $\ge n$ implies that $\mathrm{Res}^G_eZ$ and therefore $\Sigma^k\mathrm{Res}^G_eZ$ is $\ge n$ in the usual $t$-structure on $\Sp$.
\end{proof}

By definition of the slice tower, the filtered $G$-spectrum $\mathrm{Slice}(F(EG_+,X))$ is connective in the slice $t$-structure. This results in a canonical map
 \[
    \mathrm{Slice}(F(EG_+,X))\to\tau_{\ge0}^{\mathrm{slice}}\mathrm{HFP}(X),
    \]
which is an equivalence pointwise by the lemma, and we have proved the following.

\begin{theorem}\label{thm:sliceasconnectivecover}
    For any $X\in \Sp^G$, the map $\mathrm{Slice}(X)\to\mathrm{HFP}(X)$ factors canonically through an equivalence
    \[
    \mathrm{Slice}(F(EG_+,X))\xrightarrow{\simeq}\tau_{\ge0}^{\mathrm{slice}}\mathrm{HFP}(X).
    \]
       If $X$ has a naive $\mathbb{E}_\infty$-structure, this is an equivalence of $G$-$\mathbb{E}_\infty$-algebras in $\Fil(\myuline{\Sp})$.
\end{theorem}

\begin{corollary}\label{cor:borelconnectivecover}
    For a Borel-complete $G$-spectrum $X$, the slice filtration of $X$ is the connective cover of the homotopy fixed point filtration of $X$ with respect to the slice $t$-structure.
\end{corollary}

If $G$ is a $p$-group, the Segal conjecture states that the $p$-completion of any finite $G$-spectrum is Borel-complete, which gives the following.

\begin{corollary}\label{cor:segalconjecturecorollary}
    Let $X$ be the $p$-completion of a finite $G$-spectrum. Then the slice filtration of $X$ is the connective cover of the homotopy fixed-point filtration of $X$ with respect to the filtered slice $t$-structure, i.e.
    \[
   \mathrm{Slice}(X)\simeq\tau_{\ge0}^{\mathrm{slice}}\mathrm{HFP}(X).
   \]
   If $X$ has a naive $\mathbb{E}_\infty$-structure, this is an equivalence of $G$-$\mathbb{E}_\infty$-algebras in $\Fil(\myuline{\Sp})$.
\end{corollary}

In \cite{carrickcofree}, the author showed that the norms of Real bordism theory $MU^{(\!(G)\!)}$ are Borel complete, which gives the following.

\begin{corollary}\label{cor:sliceMUGisacover}
   The slice filtration of $MU^{(\!(G)\!)}$ is the connective cover of the homotopy fixed point filtration of $MU^{(\!(G)\!)}$ , i.e. there is an equivalence
   \[
   \mathrm{Slice}(MU^{(\!(G)\!)})\simeq\tau_{\ge0}^{\mathrm{slice}}\mathrm{HFP}(MU^{(\!(G)\!)}).
   \]
   of $G$-$\mathbb{E}_\infty$-algebras in $\Fil(\myuline{\Sp})$.
\end{corollary}

The following equivalence is also a useful identification.

\begin{theorem}\label{thm:slicey=xdescription}
    There is an equivalence 
    \[\tau_{\ge0}^{y=x}\mathrm{Slice}(MU^{(\!(G)\!)})^G\simeq \tau_{\ge0}^{y=x}\mathrm{HFP}(MU^{(\!(G)\!)})^G\]
    of commutative algebras in filtered spectra.
\end{theorem}
\begin{proof}
    Note that since $MU^{(\!(G)\!)}$ is Borel-complete and connective, the map
    \[\mathrm{Slice}^n(MU^{(\!(G)\!)})\to\mathrm{HFP}^n(MU^{(\!(G)\!)})\]
    is an equivalence for all $n\le 0$. By Proposition \ref{prop:sufficientforconnectivecover}, it therefore suffices to show that the map $\mathrm{Slice}(MU^{(\!(G)\!)})\to \mathrm{HFP}(MU^{(\!(G)\!)})$ induces an equivalence of filtered spectra
    \[\tau_{\ge0}^{y=x}(\mathrm{Slice}(MU^{(\!(G)\!)})^G/\tau)\to \tau_{\ge0}^{y=x}(\mathrm{Slice}(MU^{(\!(G)\!)})^G/\tau),\]
which was shown by Ullman \cite[Theorem 9.4]{ullman}.    
\end{proof}

\subsection{$MU$-synthetic lifts of the slice filtration} The slice $t$-structure allows us to implement the slice filtration as an $\mathbb{E}_\infty$-algebra in the category of synthetic spectral Mackey functors
\[
\mathrm{Syn}_{MU}^G:=\mathrm{Mack}(G;\mathrm{Syn}_{MU}).
\]
In the following, we will denote by $\mathrm{inf}_e^G$ the inflation functor $\mathrm{Fil}(\Sp)\to \mathrm{Fil}(\Sp^G)$ given by applying $\mathrm{inf}_e^G:\Sp\to\Sp^G$ pointwise to a filtered spectrum. 

\begin{proposition}\label{prop:syntheticspectralmackey}
    There is an equivalence of symmetric monoidal $\infty$-categories
    \[
    \mathrm{Syn}_{MU}^G\simeq \mathrm{Mod}_{\Fil(\Sp^G)}(\mathrm{inf}_e^G\mathrm{AN}(\mathbb{S})). 
    \]
    \end{proposition}
    \begin{proof}
    We first note that $\mathrm{Syn}_{MU}$ is compactly generated by the bigraded spheres, by \cite[Theorem 6.3]{piotr}. It follows from \cite[Observation 1.9]{AMGR} that there is a symmetric monoidal equivalence
    \[\mathrm{Syn}_{MU}^G\simeq \Sp^G\otimes \mathrm{Syn}_{MU}.\]
    We therefore define the symmetric monoidal left adjoint $f^*_G:\Fil(\Sp^G)\to \mathrm{Syn}_{MU}^G$ to be the tensor product 
    \[f^*_G:\Fil(\Sp^G)\simeq\Sp^G\otimes\Fil(\Sp)\xrightarrow{\mathrm{id}\otimes f^*}\Sp^G\otimes \mathrm{Syn}_{MU}\simeq \mathrm{Syn}_{MU}^G,\]
    where $f^*$ is the symmetric monoidal left adjoint defined by the symmetric monoidal functor $\Z\to \mathrm{Syn}_{MU}$ sending $i$ to the bigraded sphere $\mathbb{S}^{0,i}$. Let $f_*^G$ denote the right adjoint to $f^*_G$. Combining \cite[Theorem 6.3]{piotr} with \cite[Proposition A.4]{bhs2}, we see that $f_*^G$ factors through a symmetric monoidal equivalence
    \[\mathrm{Mod}_{\Fil(\Sp^G)}(f_*^G(\mathbbm{1}))\simeq\mathrm{Syn}_{MU}^G,\]
    and it remains to identify the $\mathbb{E}_\infty$-algebra $f_*^G(\mathbbm{1})$ with $\mathrm{inf}_e^G\mathrm{ANSS}(\mathbb{S})$. For this, note that since $f_*$ commutes with colimits, the right adjoint $f_*^G$ is given by the tensor product
    \[
    f_*^G:\mathrm{Syn}_{MU}^G\simeq \Sp^G\otimes \mathrm{Syn}_{MU}\xrightarrow{\mathrm{id}\otimes f_*}\Sp^G\otimes\Fil(\Sp)\simeq\Fil(\Sp^G)
    \]
    which implies a commutative diagram
    \[
    \begin{tikzcd}
        \mathrm{Syn}_{MU}\arrow[d,"f_*"]\arrow[r,"\mathrm{inf}_e^G\otimes\mathrm{id}"]&\mathrm{Syn}_{MU}^G\arrow[d,"f_*^G"]\\
        \Fil(\Sp)\arrow[r,"\mathrm{inf}_e^G\otimes\mathrm{id}"]& \Fil(\Sp^G)
    \end{tikzcd}
    \]
    The result now follows from the fact that $f_*(\mathbbm{1})\simeq \mathrm{AN}(\mathbb{S})$ (see \cite[Proposition 1.25]{CDvN}).
\end{proof}

In \cite{CDvN}, the author, Davies, and van Nigtevecht showed how to implement homotopy fixed point spectral sequences and, more generally, descent spectral sequences as commutative algebras in $\mathrm{Syn}_{MU}$. This yields the following, which we reprove in our setting.

\begin{proposition}\label{prop:hfpssisMUsynthetic}
    Let $R\in\Fun(BG,\mathrm{CAlg}(\Sp))$ have the property that $\mathrm{Res}^G_eR$ is complex-orientable. Then $\mathrm{HFP}(R)$ admits a canonical lift along the forgetful functor
    \[\mathrm{CAlg}(\mathrm{Syn}_{MU}^G)\to \mathrm{CAlg}(\Fil(\Sp^G))\]
\end{proposition}
\begin{proof}
By Proposition \ref{prop:syntheticspectralmackey}, it suffices to produce a map 
\[\mathrm{inf}_e^G\mathrm{AN}(\mathbb{S})\to \mathrm{HFP}(R)\]
in $\mathrm{CAlg}(\Fil(\Sp^G))$. Let $\mathrm{CAlg}^{\mathbb{C}}(\Sp)$ denote the full subcategory of $\mathrm{CAlg}(\Sp)$ consisting of those $R$ with the property that $R$ is complex orientable, and denote by $i$ the inclusion $\mathrm{CAlg}^{\mathbb{C}}(\Sp)\hookrightarrow \mathrm{CAlg}(\Sp)$. Consider the functor
\begin{align*}
    \mathrm{Dec}:\mathrm{CAlg}(\Sp)&\to\mathrm{CAlg}(\Fil(\Sp))\\
    R&\mapsto \mathrm{Tot}(\mathrm{Wh}(R\otimes MU^{\otimes\bullet +1}))
\end{align*}
 There is an equivalence of functors $\mathrm{Dec}\circ i\simeq \mathrm{Wh}\circ i$, obtained by combining \cite[Lemma 1.19 and Proposition 1.25]{CDvN}.
 
 Consider now the functor 
\begin{align*}
    \mathrm{Dec}^G:\mathrm{CAlg}(\Sp^G)&\to\mathrm{CAlg}(\Fil(\Sp^G))\\
    R&\mapsto \mathrm{Tot}(\mathrm{Wh}(R\otimes \mathrm{inf}_e^G(MU)^{\otimes\bullet +1}))
\end{align*}
Via the unit $\mathbb{S}\to\mathrm{inf}_e^G(MU)$, there is a map
\[
\mathrm{Wh}(R)\simeq \mathrm{Tot}(\mathrm{const}(\mathrm{Wh}(R))\to \mathrm{Tot}(\mathrm{Wh}(R\otimes \mathrm{inf}_e^G(MU)^{\otimes\bullet +1}))=\mathrm{Dec}^G(R)
\]
in $\mathrm{CAlg}(\Sp^G)$, where $\mathrm{const}(\mathrm{Wh}(R))$ is the constant cosimplicial object at $\mathrm{Wh}(R)$. We have shown this map induces an equivalence on underlying filtered spectra when $\mathrm{Res}^G_eR$ is complex-orientable, so this gives an equivalence
\[\mathrm{HFP}(R)=F({EG}_+,\mathrm{Wh}(R))\simeq F({EG}_+,\mathrm{Dec}^G(R))\]
in $\mathrm{CAlg}(\Fil(\Sp^G))$ in this case. This gives a map
\[F({EG}_+,\mathrm{Dec}^G(\mathbb{S}))\to  F({EG}_+,\mathrm{Dec}^G(R))\simeq\mathrm{HFP}(R)\]
in $\mathrm{CAlg}(\Fil(\Sp^G))$ by functoriality of $\mathrm{Dec}$.

The proof is complete if we can produce an equivalence
\[F({EG}_+,\mathrm{inf}_e^G(\mathrm{AN}(\mathbb{S})))\simeq F({EG}_+,\mathrm{Dec}^G(\mathbb{S})) \]
in $\mathrm{CAlg}(\Fil(\Sp^G))$, as the desired map is then the composite
\[\mathrm{inf}_e^G\mathrm{AN}(\mathbb{S})\to F({EG}_+,\mathrm{inf}_e^G(\mathrm{AN}(\mathbb{S})))\simeq F({EG}_+,\mathrm{Dec}^G(\mathbb{S}))\to\mathrm{HFP}(R)\]
There is a limit comparison map
\[\mathrm{inf}_e^G\mathrm{AN}(\mathbb{S})\to \mathrm{Tot}(\mathrm{inf}_e^G(\mathrm{Wh}(MU^{\otimes\bullet+1})))\]
which induces the identity on underlying filtered spectra. For $R\in\mathrm{CAlg}(\Sp)$, there is a natural map
\[\mathrm{inf}_e^G(\mathrm{Wh}(R))\to\mathrm{Wh}(\mathrm{inf}_e^G(R))\]
in $\mathrm{CAlg}(\Fil(\Sp^G))$, which induces the identity on underlying filtered spectra. This is given by the definition of $\mathrm{Wh}(\mathrm{inf}_e^G(R))$ as a connective cover of $\mathrm{const}(\mathrm{inf}_e^G(R))$ (Definition \ref{def:whitehead}), the
inflated map
\[\mathrm{inf}_e^G(\mathrm{Wh}(R))\to\mathrm{inf}_e^G(\mathrm{const}(R))\simeq \mathrm{const}(\mathrm{inf}_e^G(R)),\]
and the fact that $\mathrm{inf}_e^G(\tau_{\ge n}R)$ is $n$-connective in $\Sp^G$. These maps determine a natural map
\[\mathrm{Tot}(\mathrm{inf}_e^G(\mathrm{Wh}(MU^{\otimes\bullet+1})))\to \mathrm{Tot}(\mathrm{Wh}(\mathrm{inf}_e^G(MU)^{\otimes\bullet+1}))\]
in $\mathrm{CAlg}(\Fil(\Sp^G))$ which induces the identity on underlying filtered spectra and therefore induces an equivalence on Borelification, completing the proof.
\end{proof}

We will combine Proposition \ref{prop:hfpssisMUsynthetic} and Corollary \ref{cor:sliceMUGisacover} to produce a map
\[\mathrm{inf}_e^G\mathrm{AN}(\mathbb{S})\to \mathrm{Slice}(MU^{(\!(G)\!)})\]
in $\mathrm{CAlg}(\Fil(\Sp^G))$. First we need to trade the slice $t$-structure for the $\mathcal{O}$-slice $t$-structure of Definition \ref{def:twofilteredtstructures} in the statement of Corollary \ref{cor:sliceMUGisacover}, since $\mathrm{inf}_e^G\mathrm{AN}(\mathbb{S})$ is not connective in the slice $t$-structure, but it is connective in the $\mathcal{O}$-slice $t$-structure.

\begin{lemma}\label{lem:tstructuresagreeonhfpss}
    For $X\in\Sp^G$, there is a natural equivalence
    \[\tau_{\ge0}^{\mathrm{slice}}\mathrm{HFP}(X)\simeq \tau_{\ge0}^{\mathcal{O}}\mathrm{HFP}(X)\]
    in $\Fil(\Sp^G)$. If $X$ has a naive $\mathbb{E}_\infty$-structure, this is an equivalence of $\mathcal{O}$-commutative monoids in $\Fil(\myuline{\Sp})$.
\end{lemma}
\begin{proof}
    By definition of these $t$-structures, if $Y\in\Fil(\Sp)$ is connective in the slice $t$-structure, then it is connective in the $\mathcal{O}$-slice $t$-structure. This gives a map
\begin{equation}\label{eq:comparisonofcovers}
       \tau_{\ge0}^{\mathrm{slice}}\mathrm{HFP}(X)\to \tau_{\ge0}^{\mathcal{O}}\mathrm{HFP}(X), 
    \end{equation}
   which is a map of $\mathcal{O}$-commutative monoids in $\Fil(\myuline{\Sp})$ when $X$ has a naive $\mathbb{E}_\infty$-structure, and we claim this map is an equivalence. The map (\ref{eq:comparisonofcovers}) is a slice connective cover, so
   \[[Z,C]_{\Fil(\Sp^G)}=0\]
 for all generators $Z$ of the slice $t$-structure, where $C$ is the cofiber of the map (\ref{eq:comparisonofcovers}). Since both the domain and codomain of (\ref{eq:comparisonofcovers}) are $\mathcal{O}$-slice connective, it suffices to show that the map induces an equivalence on $\mathcal{O}$-slice connective covers, i.e. that $[W,C]_{\Fil(\Sp^G)}=0$ for all generators $W$ of the $\mathcal{O}$-slice $t$-structure. A generator $W$ of the $\mathcal{O}$-slice $t$-structure is either of the form $\mathrm{Ind}_e^G(Y)$ for $Y\in\Fil(\Sp)$ or  $\tilde{E}G\otimes Z$ for $Z$ a generator of the slice $t$-structure. We claim that $C\simeq \tilde{E}G\otimes C$, and this completes the proof, as
 \begin{align*}
    [\mathrm{Ind}_e^G(Y),C]_{\Fil(\Sp^G)}&\cong [\mathrm{Ind}_e^G(Y),\tilde{E}G\otimes C]_{\Fil(\Sp^G)}\\
    &\cong[Y,\mathrm{Res}_e^G(\tilde{E}G\otimes C)]_{\Fil(\Sp)}\\
    &=0 
 \end{align*}

   and
   \begin{align*}
   [\tilde{E}G\otimes Z,C]_{\Fil(\Sp^G)}&\cong[\tilde{E}G\otimes Z,\tilde{E}G\otimes C]_{\Fil(\Sp^G)}\\
   &\cong[Z,\tilde{E}G\otimes C]_{\Fil(\Sp^G)}\\
   &\cong[Z,C]_{\Fil(\Sp^G)}\\
   &=0    
   \end{align*}

   To prove the claim, note that the composition
   \[\tau_{\ge0}^{\mathrm{slice}}\mathrm{HFP}(X)\to \tau_{\ge0}^{\mathcal{O}}\mathrm{HFP}(X)\to \mathrm{HFP}(X)\]
   gives a cofiber sequence
   \[C\to \tau_{<0}^{\mathrm{slice}}\mathrm{HFP}(X)\to \tau_{<0}^{\mathcal{O}}\mathrm{HFP}(X)\]
   so it suffices to prove that $\tau_{<0}^{\mathrm{slice}}\mathrm{HFP}(X)\simeq \tilde{E}G\otimes\tau_{<0}^{\mathrm{slice}}\mathrm{HFP}(X)$ and $\tau_{<0}^{\mathcal{O}}\mathrm{HFP}(X)\simeq \tilde{E}G\otimes\tau_{<0}^{\mathcal{O}}\mathrm{HFP}(X)$. By checking on generators, applying $\mathrm{Res}^G_e$ to the slice $t$-structure gives the $y=0$ linear $t$-structure on $\Fil(\Sp)$, and applying $\mathrm{Res}^G_e$ to the $\mathcal{O}$-slice $t$-structure gives the trivial $t$-structure on $\Fil(\Sp)$, i.e. the $t$-structure where all objects are connective. The filtered spectrum $\mathrm{Res}^G_e\mathrm{HFP}(X)=\mathrm{Wh}(\mathrm{Res}^G_eX)$ is connective in both these $t$-structures, so the claim follows.
\end{proof}

We now have all the pieces in place to prove our main theorem.

\begin{theorem}\label{thm:mapfrominflationtoslice}
    Let $R\in\mathrm{CAlg}(\Sp^G)$ have the property that $\mathrm{Res}^G_eR$ is complex-orientable. Then $\mathrm{Slice}(F({EG}_+,R))$ admits a canonical lift along the forgetful functor
    \[\mathrm{CAlg}(\mathrm{Syn}_{MU}^G)\to \mathrm{CAlg}(\Fil(\Sp^G))\]
\end{theorem}
\begin{proof}
    By Proposition \ref{prop:hfpssisMUsynthetic}, there is a map 
    \[\mathrm{inf}_e^G\mathrm{AN}(\mathbb{S})\to \mathrm{HFP}(R)\]
    in $\mathrm{CAlg}(\Fil(\Sp^G))$, and combining Lemma \ref{lem:tstructuresagreeonhfpss} with Theorem \ref{thm:sliceasconnectivecover}, we have that $\mathrm{Slice}(F({EG}_+,R))$ is the $\mathcal{O}$-slice connective cover of $\mathrm{HFP}(R)$. It therefore suffices to show that $\mathrm{inf}_e^G\mathrm{AN}(\mathbb{S})$ is $\mathcal{O}$-slice connective.

    We claim that the filtered spectrum $\mathrm{AN}(\mathbb{S})$ is connective in the $y=x$ linear $t$-structure on $\Fil(\Sp)$. Given this, it follows that $\mathrm{AN}^n(\mathbb{S})$ is contained in the subcategory $\langle\mathbb{S}^{\lceil\frac{n}{2}\rceil}\rangle$ of $\Sp$. Since $\mathrm{inf}_e^G$ is a left adjoint, $\mathrm{inf}_e^G\mathrm{AN}^n(\mathbb{S})$ is in the subcategory $\langle\mathbb{S}^{\lceil\frac{n}{2}\rceil}\rangle$ of $\Sp^G$. The $G$-spectrum $EG_+\otimes\mathrm{inf}_e^G\mathrm{AN}^n(\mathbb{S})$ is in the subcategory of $\Sp^G$ generated by the induced spheres $\mathrm{Ind}_e^G\mathbb{S}^i$, via the skeletal filtration of $EG_+$, and hence $EG_+\otimes\mathrm{inf}_e^G\mathrm{AN}^n(\mathbb{S})$ is connective in the $n$-th $t$-structure in the $\mathcal{O}$-slice filtration. Via isotropy separation, it suffices to show that $\tilde{E}G\otimes\mathrm{inf}_e^G\mathrm{AN}^n(\mathbb{S})$ is connective in the $n$-th $t$-structure in the $\mathcal{O}$-slice filtration. The $G$-spectrum $\tilde{E}G\otimes \mathbb{S}^{\lceil\frac{n}{2}\rceil}$ is contained in $\langle\mathbb{S}^{m\rho_H}:m|H|\ge n\rangle_{H\subset G}$ by applying the geometric connectivity criterion of \cite[Theorem 11.1.27]{HHRbook}, so $\tilde{E}G\otimes \mathbb{S}^{\lceil\frac{n}{2}\rceil}$ is contained in $\langle\tilde{E}G\otimes\mathbb{S}^{m\rho_H}:m|H|\ge n\rangle_{H\subset G}$. This implies now that $\tilde{E}G\otimes\mathrm{inf}_e^G\mathrm{AN}^n(\mathbb{S})$ is contained in $\langle\tilde{E}G\otimes\mathbb{S}^{m\rho_H}:m|H|\ge n\rangle_{H\subset G}$ as desired.

    The fact that $\mathrm{AN}(\mathbb{S})$ is connective in the $y=x$ linear $t$-structure on $\Fil(\Sp)$ boils down to the well known fact that the Adams--Novikov spectral sequence has a $y=x$ vanishing line on the $E_2$-page. Setting $(A,\Gamma)=(MU_*,MU_*MU)$, the $E_2$-page is the cohomology of the cobar complex
    \[A\implies\overline{\Gamma}\Rrightarrow\overline{\Gamma}\otimes_A \overline{\Gamma}\cdots\]
    where $\overline{\Gamma}=\ker(\epsilon:\Gamma\to A)$, by \cite[Corollary A1.2.12]{ravgreen}. Since $\overline{\Gamma}$ is concentrated in degrees $\ge 2$, the $s$-th term $\overline{\Gamma}^{\otimes_As}$ of this cobar complex is concentrated in degrees $\ge 2s$, and it follows that $\Ext_{(A,\Gamma)}^{s,t}(A,A)$ is concentrated in degres $t\ge 2s$, i.e. $s\le t-s$. This implies that $\mathrm{AN}(\mathbb{S})/\tau$ is connective in the $y=x$ linear $t$-structure on $\Fil(\Sp)$, by the identification $\pi_{n,s}\mathrm{AN}(\mathbb{S})/\tau\cong \Ext_{(A,\Gamma)}^{s,n+s}(A,A)$ of \cite[Proposition 3.3]{mmf}. The filtered spectrum $\mathrm{AN}(\mathbb{S})$ is complete because it is defined as a limit of Whitehead towers, which are complete. Moreover, it is strongly complete in the sense of Definition \ref{def:stronglycomplete} since $\mathrm{ANSS}(\mathbb S)$ converges strongly, as shown in \cite[Chapter 2, Section 2]{ravgreen}. Proposition \ref{prop:connectivitymodtau} now implies that
    $\mathrm{AN}(\mathbb{S})$ is also connective in the $y=x$ linear $t$-structure. 
\end{proof}

We have the following concrete consequence for the spectral sequences corresponding to these filtrations.

\begin{corollary}
    Let $R\in\mathrm{CAlg}(\Sp^G)$ have the property that $\mathrm{Res}^G_eR$ is complex-orientable and $R$ is Borel complete. There is a map of multiplicative spectral sequences 
    \[
    \mathrm{ANSS}(\mathbb{S})\to \mathrm{SliceSS}(R)
    \]
    converging to the unit map
    \[
    \pi_*\mathbb{S}\to\pi_*^GR
    \]
\end{corollary}

This is only a shadow of the structure given by the $\mathbb{E}_\infty$ map
    \[
    \mathrm{AN}(\mathbb{S})\to \mathrm{Slice}(R).
    \]
The map on spectral sequences also respects all higher structures furnished by the $\mathbb{E}_\infty$ structures, such as Toda brackets and power operations. Moreover, since $\mathrm{Slice}(R)$ lifts to $\mathrm{Syn}_{MU}^G$, for any $R$-module $M$, $\mathrm{Slice}(M)$ also lifts to $\mathrm{Syn}_{MU}^G$. We give the following example, which was our motivating example.

\begin{corollary}\label{cor:slicebpgm}
    There is a map of spectral sequences
    \[\mathrm{ANSS}(\mathbb{S})\to\mathrm{SliceSS}(\bpgm{m}),\]
    which is a map of modules over $\mathrm{ANSS}(\mathbb{S})$ converging to the unit map 
        \[
\pi_*\mathbb{S}\to\pi_*^G\bpgm{m}.
    \]
\end{corollary}

\begin{remark}\label{rmk:RO(G)gradings}
    In the above results on spectral sequences, we have stated things in terms of integer gradings. However, these statements all hold in full $RO(G)$-grading by replacing $\mathrm{ANSS}(\mathbb{S})$ with the $\mathrm{inf}_e^G(MU)$-based ASS of the equivariant sphere spectrum $\mathbb S$.
\end{remark}

\subsection{$MU_G$-synthetic lifts of the slice filtration} The results of the previous subsection give us a way to produce maps into the slice filtration, by first producing a map to the homotopy fixed point filtration and then checking a connectivity condition. We apply this now to the (integer-graded) $MU_G$-based ASS.

\begin{proposition}\label{prop:hfpssisMUGsynthetic}
    Let $R\in\Fun(BG,\mathrm{CAlg}(\Sp))$ have the property that $\mathrm{Res}^G_eR$ is complex-orientable. Then there is a map
    \[
    \mathrm{AN}_G(\mathbb{S})\to \mathrm{HFP}(R)
    \]
    in $\mathrm{CAlg}(\Fil(\Sp^G))$.
\end{proposition}
\begin{proof}
    As in the proof of Proposition \ref{prop:hfpssisMUsynthetic}, we have a map
\[\mathrm{Wh}(R)=\mathrm{Tot}(\mathrm{const}(\mathrm{Wh}(R)))\to \mathrm{Tot}(\mathrm{const}(\mathrm{Wh}(MU_G^{\otimes\bullet+1}\otimes R))),\]
which induces an equivalence on underlying filtered spectra. This gives a diagram
\[
\begin{tikzcd}
&&\mathrm{AN}_G(\mathbb{S})\arrow[dddll,dashed]\arrow[d,equal]\\
&&\mathrm{Tot}(\mathrm{Wh}(MU_G^{\otimes\bullet+1}))\arrow[d]\\
    &&\mathrm{Tot}(\mathrm{Wh}(MU_G^{\otimes\bullet+1}\otimes R))\arrow[d]\\
    \mathrm{HFP}(R)\arrow[r,equal]&F({EG}_+,\mathrm{Wh}(R))\arrow[r,"\simeq"]& F({EG}_+,\mathrm{Tot}(\mathrm{Wh}(MU_G^{\otimes\bullet+1}\otimes R)))
\end{tikzcd}\]
and the dotted arrow completes the proof.
\end{proof}

We now let $G$ be a cyclic $2$-group, as in the context of Conjecture \ref{conjectureweak}.

\begin{theorem}\label{thm:integersliceisMUGsyntheticgivenconjecture}
    Assuming Conjecture \ref{conjectureweak}, there is a map of $\mathcal{O}$-algebras
    \[\mathrm{AN}_G(\mathbb{S})\to\mathrm{Slice}(MU^{(\!(G)\!)})\]
    in $\Fil(\Sp^G)$.
\end{theorem}
\begin{proof}
    Combining Proposition \ref{prop:hfpssisMUGsynthetic} with Lemma \ref{lem:tstructuresagreeonhfpss} and Theorem \ref{thm:sliceasconnectivecover}, it suffices to show that $\mathrm{AN}_G(\mathbb{S})$ is connective in the $\mathcal{O}$-slice $t$-structure on $\Fil(\Sp^G)$. Using Proposition \ref{prop:connectivitymodtau} it suffices to establish this on associated graded, and the conjecture states that the filtered spectrum $\mathrm{AN}_G(\mathbb S)^H$ is connective in the $y=(|H|-1)x$ linear $t$-structure, for all nontrivial subgroups $H$. Since, for a family of subgroups $\mathcal F$ of $G$, the $G$-spectrum $\tilde{E}\mathcal{F}$ is connective, it follows that the filtered spectrm $\Phi^H\mathrm{AN}_G(\mathbb S)$ is also connective in the $y=(|H|-1)x$ linear $t$-structure, for all nontrivial subgroups $H$. This completes the proof by Remark \ref{rmk:geometricconnectivity}.
    \end{proof}

\begin{remark}
    The $\mathcal{O}$-algebra structure on the map of Theorem \ref{thm:integersliceisMUGsyntheticgivenconjecture} implies that corresponding map of spectral sequences respects norms from nontrivial subgroups. It is possible to make this precise, e.g. in the framework of the spectral sequences with norm structure of Meier--Shi--Zeng \cite{msz}, but we do not pursue this here.
\end{remark}

\subsection{$G$-commutative $MU_G$-synthetic lifts of the slice filtration} We show that, given our conjecture on the $MU_G$-based Adams spectral sequence, there is an analogue of Theorem \ref{thm:mapfrominflationtoslice} in this context. In the following, we let $F({EG}_+,-):\Fun(RO(G)^{\mathrm{op}},\Sp^G)\to\Fun(RO(G)^{\mathrm{op}},\Sp^G)$ be the functor given by applying $F({EG}_+,-)$ pointwise.

\begin{definition}\label{def:ROhfpsss}
    The $\myuline{RO}$-indexed homotopy fixed point filtration $\mathrm{HFP}^{RO}(X)$ of $X\in\Sp^G$ is the image of $X$ under the composite
    \[\Sp^G\xrightarrow{\mathrm{Wh}^{RO}}\Fil(\Sp^G)\xrightarrow{F({EG}_+,-)}\Fil(\Sp^G),\]
    where $\mathrm{Wh}^{RO}$ denotes the Whitehead tower functor of Definition \ref{def:ROwhitehead}.
\end{definition}

\begin{proposition}\label{prop:ROhfpispulledbackalongdim}
    For $X\in\Sp^G$, there is a canonical equivalence
    \[\mathrm{HFP}^{RO}(X)\simeq \mathrm{dim}^*\mathrm{HFP}(X).\]
    If $X$ has a naive $\mathbb{E}_\infty$ structure, this is an equivalence in $\mathrm{CAlg}_G(\Fun(\myuline{RO}^{\mathrm{op}},\myuline{\Sp}))$.
\end{proposition}
\begin{proof}
    The filtration $\mathrm{Wh}^{RO}(X)$ is obtained by applying the connective cover functor to $\mathrm{const}(X)$ with respect to the $\myuline{RO}$-indexed slice filration on $\Sp^G$, generated at the pair $(H,V\in RO(H))$ by $\{\Sigma^V\mathrm{Ind}_K^H\mathbb{S}\}_{K\subset H}$. On the other hand, $\mathrm{dim}^*\mathrm{Wh}(X)$ is seen to be obtained by applying the connective cover functor to $\mathrm{const}(X)$ with respect to the $\myuline{RO}$-indexed slice filration on $\Sp^G$, generated at the pair $(H,V\in RO(H))$ by $\{\Sigma^{|V|}\mathrm{Ind}_K^H\mathbb{S}^{}\}_{K\subset H}$. One may consider a mixed filtration given by the $\myuline{RO}$-indexed slice filration on $\Sp^G$, generated at the pair $(H,V\in RO(H))$ by the union of the two above generating sets. This gives a natural diagram
    \[\mathrm{Wh}^{RO}(X)\to\mathrm{Mixed}(X)\leftarrow\mathrm{dim}^*\mathrm{Wh}(X),\]
    which induces the identity on underlying $RO(G)$-filtered spectra and therefore equivalences on Borelification.
\end{proof}

In light of the previous proposition, we make the following definition.

\begin{definition}\label{def:ROslice}
    For $X\in\Sp$, let
    \[\mathrm{Slice}^{RO}(X):=\mathrm{dim}^*\mathrm{Slice}(X)\in\Fun(\myuline{RO}^{\mathrm{op}},\myuline{\Sp}).\]
\end{definition}

\begin{remark}\label{rmk:definitionofROslice}
    This definition of an $\myuline{RO}$-indexed slice filtration may also be obtained as follows. One defines the $\myuline{RO}$-indexed slice filtration on $\myuline{\Sp}$ with $t$-structure at the pair $(H,V\in RO(H))$ given by the category of slice $|V|$-connective $H$-spectra, i.e. the $t$-structure generated by $\{\mathrm{Ind}_L^G\mathbb{S}^{m\rho_L}:m|L|\ge |V|\}$, as in Definition \ref{def:twofilteredtstructures}. Then, one takes the connective cover of $\mathrm{const}(X)$ with respect to the corresponding $t$-structure on $\Fun(\myuline{RO}^{\mathrm{op}},\myuline{\Sp})$ to obtain $\mathrm{Slice}^{RO}(X)$. This correctly implements the $RO(G)$-graded slice spectral sequence in the usual way, since the $E_2$-page of the latter is defined as $\pi_{V-s}\mathrm{gr}^{|V|}\mathrm{Slice}(X)$.
\end{remark}

It is straightforward now to give a $G$-$\mathbb{E}_\infty$ map $\mathrm{AN}^{RO}_G(\mathbb{S})\to \mathrm{HFP}^{RO}(R)$ when $R$ is a naive $\mathbb{E}_\infty$-algebra with $\mathrm{Res}^G_eR$ complex-orientable, along the lines of Proposition \ref{prop:hfpssisMUsynthetic}. This gives $\mathrm{HFP}^{RO}(R)$ an $MU_G$-synthetic structure, according to the following \emph{ad hoc} definition, along the lines of \cite{mmf}.

\begin{definition}\label{def:MUGsynthetic}
    Define the $G$-symmetric monoidal $\infty$-category of $MU_G$-synthetic spectra by
    \[\mathrm{Syn}_{MU_G}=\mathrm{Mod}_{\Fun(\underline{RO}^{\mathrm{op}},\myuline{\scriptstyle{\Sp}})}(\mathrm{AN}^{RO}_G(\mathbb{S})).\]
\end{definition}

\begin{proposition}\label{prop:ROhfpssisMUsynthetic}
    Let $R\in\Fun(BG,\mathrm{CAlg}(\Sp))$ have the property that $\mathrm{Res}^G_eR$ is complex-orientable. Then $\mathrm{HFP}^{RO}(R)$ admits a canonical lift along the forgetful functor
    \[\mathrm{CAlg}_G(\mathrm{Syn}_{MU_G})\to \mathrm{CAlg}_G(\Fil(\Sp^G)).\]
\end{proposition}
\begin{proof}
It suffices to produce a map 
\[\mathrm{AN}_G^{RO}(\mathbb{S})\to \mathrm{HFP}^{RO}(R)\]
in $\mathrm{CAlg}_G(\Fil(\myuline{\Sp}))$. As in the proof of Proposition \ref{prop:hfpssisMUsynthetic}, we have a map
\[\mathrm{Wh}^{RO}(R)=\mathrm{Tot}(\mathrm{const}(\mathrm{Wh}^{RO}(R)))\to \mathrm{Tot}(\mathrm{const}(\mathrm{Wh}^{RO}(MU_G^{\otimes\bullet+1}\otimes R))),\]
which induces an equivalence on underlying $RO(G)$-filtered spectra. This gives a diagram
\[
\begin{tikzcd}
&&\mathrm{AN}^{RO}_G(\mathbb{S})\arrow[dddll,dashed]\arrow[d,equal]\\
&&\mathrm{Tot}(\mathrm{Wh}^{RO}(MU_G^{\otimes\bullet+1}))\arrow[d]\\
    &&\mathrm{Tot}(\mathrm{Wh}^{RO}(MU_G^{\otimes\bullet+1}\otimes R))\arrow[d]\\
    \mathrm{HFP}^{RO}(R)\arrow[r,equal]&F({EG}_+,\mathrm{Wh}^{RO}(R))\arrow[r,"\simeq"]& F({EG}_+,\mathrm{Tot}(\mathrm{Wh}^{RO}(MU_G^{\otimes\bullet+1}\otimes R)))
\end{tikzcd}\]
and the dotted arrow completes the proof.
\end{proof}

By adjunction and use of Proposition \ref{prop:ROhfpispulledbackalongdim}, the map of $G$-$\mathbb{E}_\infty$ rings $\mathrm{AN}^{RO}_G(\mathbb{S})\to \mathrm{HFP}^{RO}(R)$ of Proposition \ref{prop:ROhfpssisMUsynthetic} gives a map of $G$-$\mathbb{E}_\infty$ rings $\mathrm{Total}(\mathrm{AN}^{RO}_G(\mathbb{S}))\to \mathrm{HFP}(R)$ in $\Fil(\myuline{\Sp})$. This gives now the following, where as before $\mathcal{O}$ is the $\mathbb{N}_\infty$-operad corresponding to the transfer system with all norms $H\to H'$ for $\{e\}\neq H\subset H'\subset G$, and $\mathrm{CAlg}_{\mathcal{O}}$ denotes the category of $\mathcal{O}$-algebras.

\begin{theorem}\label{thm:conditionalonconjecture}
    Let $R\in\Fun(BG,\mathrm{CAlg}(\Sp))$ have the property that $\mathrm{Res}^G_eR$ is complex-orientable. Assuming Conjecture \ref{conjecturemain}, $\mathrm{Slice}^{RO}(F({EG}_+,R))$ admits a canonical lift along the forgetful functor
    \[\mathrm{CAlg}_{\mathcal{O}}(\mathrm{Syn}_{MU_G})\to \mathrm{CAlg}_{\mathcal{O}}(\Fun(\myuline{RO}^{\mathrm{op}},\myuline{\Sp})).\]
\end{theorem}
\begin{proof}
    By adjunction it suffices to produce a map of $\mathcal{O}$-algebras
    \[\mathrm{Total}(\mathrm{AN}^{RO}_G(\mathbb{S}))\to \mathrm{Slice}(F({EG}_+,R)).\]
    The conjecture states that $\mathrm{Total}(\mathrm{AN}^{RO}_G(\mathbb{S}))$ is connective in the $\mathcal{O}$-slice $t$-structure, so by Lemma \ref{lem:tstructuresagreeonhfpss}, it suffices to produce a map of $\mathcal{O}$-algebras
        \[\mathrm{Total}(\mathrm{AN}^{RO}_G(\mathbb{S}))\to \mathrm{HFP}(R),\]
        which is given by Proposition \ref{prop:ROhfpssisMUsynthetic}.
\end{proof}

\printbibliography

\end{document}